\documentclass[final,leqno,onefignum,onetabnum]{siamart220329}

\usepackage{graphicx}
\graphicspath{{figures/}}

\usepackage{amssymb}
\usepackage{latexsym,amsmath,amsfonts,amscd}
\usepackage{xcolor}

\usepackage[export]{adjustbox}
\usepackage{algorithm} 
\usepackage{algorithmic}
\usepackage[algo2e,linesnumbered,ruled,vlined]{algorithm2e}

\usepackage{array}
\usepackage[caption=false,font=footnotesize]{subfig}
\usepackage{lipsum}
\usepackage{hyperref}
\usepackage{comment}
\usepackage{placeins}

\usepackage{graphicx} %
\usepackage{booktabs} %
\usepackage[normalem]{ulem}
\usepackage{multirow}

\usepackage{epsfig}
\usepackage{changebar}
\usepackage{pstricks}
\usepackage{pst-plot}
\usepackage{multirow}
\usepackage{placeins}
\usepackage{mathrsfs}
\usepackage{bm,color}
\usepackage{enumerate}

\newtheorem{thm}{Theorem}[section]

\newtheorem{cor}{Corollary}[section]

\DeclareMathOperator{\He}{He}

\newcommand{\beq}{\begin{equation}}
\newcommand{\eeq}{\end{equation}}

\headers{Hyperbolic Machine Learning Moment Closures}
{A. J. Christlieb, M. Ding, J. Huang, and N. A. Krupansky}

\title{Hyperbolic Machine Learning Moment Closures for the BGK Equations \thanks{Submitted to the editors DATE. \funding{This work was funded in part by the US Department of Energy DE-SC0023164, the US National Science Foundation NSF-2309655 and NSF-2008004, and the US Department of Defense Science, Mathematics, and Research for Transformation (SMART) Scholarship, funded by OUSD/R\&E (The Under Secretary of Defense-Research and Engineering), National Defense Education Program (NDEP) / BA-1, Basic Research.}}} %

\author{Andrew J. Christlieb\footnotemark[2] \footnotemark[4]
\and Mingchang Ding\thanks{Department of Computational Mathematics, Science, and Engineering, Michigan State University, East Lansing, MI (\email{dingmin2@msu.edu})}
\and Juntao Huang\thanks{Department of Mathematics and Statistics, Texas Tech University, Lubbock, TX (\email{juntao.huang@ttu.edu})}
\and Nicholas A. Krupansky\thanks{Department of Mathematics, Michigan State University, East Lansing, MI (\email{christli@msu.edu}, \email{krupans4@msu.edu})}}

\begin{document}
\maketitle %
\begin{abstract}
We introduce a hyperbolic closure for the Grad moment expansion of the Bhatnagar-Gross-Krook's (BGK) kinetic model using a neural network (NN) trained on BGK's moment data. This closure is motivated by the exact closure for the free streaming limit that we derived in our paper on closures in transport \cite{Huang2022-RTE1}.  The exact closure relates the gradient of the highest moment to the gradient of four lower moments. As with our past work, the model presented here learns the gradient of the highest moment in terms of the coefficients of gradients for all lower ones.  By necessity, this means that the resulting hyperbolic system is not conservative in the highest moment. For stability, the output layers of the NN are designed to enforce hyperbolicity and Galilean invariance. This ensures the model can be run outside of the training window of the NN. Unlike our previous work on radiation transport that dealt with linear models, the BGK model's nonlinearity demanded advanced training tools. These comprised an optimal learning rate discovery, one cycle training, batch normalization in each neural layer, and the use of the \texttt{AdamW} optimizer.  To address the non-conservative structure of the hyperbolic model, we adopt the FORCE numerical method to achieve robust solutions. This results in a comprehensive computing model combining learned closures with methods for solving hyperbolic models. The proposed model generalizes beyond the training time window for low to moderate Knudsen numbers. Our paper details the multi-scale model construction and is run on a range of test problems.

\end{abstract}

\begin{keywords}
BGK equation, moment closure, machine learning, neural network, hyperbolicity
\end{keywords}

\begin{MSCcodes}
82C40, 82C32, 82C70
\end{MSCcodes}

\section{Introduction}
\label{sec:Introduction}

Kinetic models are important in many of the physical sciences and engineering applications when a physical system is far away from a local equilibrium, such is the case in micro-electromechanical system (MEMS) \cite{Naris2004}, atmospheric re-entry \cite{Tseng2006}, and controlled nuclear fusion \cite{Stangeby1990} to name a few. When describing systems of interest these models are plagued by the so-called ``curse of dimensionality'': in one-spatial dimension, one-velocity dimension, and time evolution we have a problem in $\mathbb{R}^3$, and in three spatial dimensions this associated model is in $\mathbb{R}^7$ phase space for each particle. This leads to the use of distributions which are further complicated by their multi-scale nature, where the behavior of solutions is dependent on the characteristic length of collisions involved. Resolving features in detail on realistic applications can quickly reach resource limits in computation with respect to discretization.

The method of moments is an approach to reduce the dimensionality in kinetic models. The first such technique was by Chapman-Enksog and Hilbert expansion, which suffered from singular moments after the first few. The ``modern'' example is due to Grad \cite{Grad1949}, 
where the distribution function in kinetic equations is expanded by a finite number of moments by integrating over velocity.
Expanding to a fixed number of moments, say the $k$-th moment, we require information of the next highest moment, the $(k+1)$-th moment, which results in an unclosed system. This leads to the so-called moment closure problem which addresses how to handle moments beyond the fixed number desired. However, Grad's moment system and closure suffer from a loss of hyperbolicity. This deficiency drew renewed attention by Levermore \cite{Levermore1996} and has driven contributions to moment closures in kinetic problems too numerous to properly exposition here. We highlight some recent work focused on closures that preserve properties of the moment system, such as hyperbolicity. One such approach is the Hyperbolic Moment Equations (HME) \cite{Cai2013,Cai2015}, where Cai \textit{et al.}\ utilize independence between moment expansion coefficients and the Jacobian's characteristic polynomial from Grad's moment system, leading to a regularization that gives global hyperbolicity. Another is the Quadrature Based Moment Equations (QBME) \cite{Koellermeier2014,Koellermeier2014-AIPConf,Koellermeier2018-QBME2d} where hyperbolicity is preserved through the use of quadrature-based projection methods.

Recent efforts in solving the closure problem employ machine learning (ML) techniques that preserve desirable qualities of the kinetic and transport systems. The pioneering work by Han \textit{et al.}\  \cite{Han-WeinanE-2019} demonstrated the utility of trained neural networks on high-fidelity kinetic data to learn closures over a wide range of Knudsen numbers, with many follow-works that preserve symmetries and invariances. Li \textit{et al.}\  \cite{Li2023} also employed neural networks, specifically the U-net architecture, to train closure models that preserve Galilean, reflecting, and scaling invariance. Charalampopoulos \textit{et al.}\  applied recurrent neural networks to correct the quadrature-based moment methods to improve accuracy in the unclosed moment \cite{Charalampopoulos2022}.

Machine learning methods have also been applied to kinetic models in other ways than the moment closure problem. Miller \textit{et al.}\ applied neural networks to approximate the collision operator of the Boltzmann equation \cite{MILLER2022111541}. Neural networks have also been used to represent the Boltzmann equation, in a sparse manner by Li \textit{el al.}\ \cite{li2023solving}. 

Hyperbolicity-preserving neural networks have been developed and applied to the closure problem. In the Radiative Transfer Equation (RTE), the work of Huang \textit{et al.}\  \cite{Huang2022-RTE1,Huang-RTE2,Huang-RTE3}, combined the structure of the closure system with a neural network, learning coefficients on the gradients of the moments. A key feature of RTE work is that the model is a local closure, while previous works were global. This is useful in large-scale, distributed computing. Schotth{\"o}fer \textit{et al.}\  \cite{pmlr-v162-schotthofer22a,schotthofer2022neural,porteous2021data} use neural network learned closures that embed entropy convexity while preserving minimum entropy and hyperbolicity for the Boltzmann equation.

By using neural networks trained on kinetic model data, there is a potential to achieve closures that are more accurate by incorporating kinetic effects instead of relying on the strong assumption of local equilibrium where one can use periodic, microscopic simulations to derive local equations of state (as in the case of the SESAME database \cite{sesame-report,sesame-McHardy2018}). Our work builds on the RTE neural network closure efforts of Huang \textit{et al.}\  \cite{Huang2022-RTE1, Huang-RTE2, Huang-RTE3} and applies it to the Bhatnagar-Gross-Krook (BGK) model of the Boltzmann kinetic equation. An analytical closure was developed for the free streaming limit in the first effort \cite{Huang2022-RTE1}. This closure relates the gradient of the highest moment to the next four lower moments. The follow-up RTE efforts \cite{Huang-RTE2,Huang-RTE3} develop provably hyperbolic closures relating many gradients of lower moments to the gradient of the highest moment. Our current work follows in that progression and by necessity, while the system is hyperbolic it is not conservative in the highest moment. Hence, for nonlinear problems such as the BGK model, we develop path-conservative numerical methods for solving the fluid ML closure \cite{castro2017well}. {  
In the numerical examples, we test two sets of ML moment closures. 
The first set is to recover the HME model using neural networks. To be more specific, we train the neural network using the data generated from the HME model. The purpose is to validate the correctness of the algorithm and code. The numerical results show good agreement between HME and our trained closures for smooth and discontinuous initial conditions across Knudsen number regimes.
The second set is to train the neural network using the data from the kinetic model. The goal is to capture the kinetic effects in the ML closure.
We have promising results for kinetic trained closure on smooth data up to transitional and free streaming (Knudsen numbers between 0.1 and 10).
}
Our contribution is demonstrating hyperbolicity preserving, local closures from neural networks that have predictive power beyond the training time window across a range Knudsen numbers.

{  Although the current work is limited to one spatial dimension and one velocity dimension (1D1V), our proposed approach to moment closure applies to current problems in inertial confinement fusion which use Vlasov–Fokker–Planck equation in 1D2V \cite{Taitano2018,Taitano2021}. Surrogate models are needed for these Vlasov-Fokker-Planck models in order to perform sensitivity studies ahead high-fidelity computations. Future work could focus on extending our methods beyond 1D1V models.}

The rest of the paper is organized as follows. In \Cref{sect:Moment Closure for the BGK Equation}, we outline the moment expansion and closure system for the BGK equations. Then, in \Cref{sec:Training Setup and Neural Network Architecture}, we outline the neural network structure, training data, and training method. In \Cref{sec:Numerical Method}, we discuss path-conservative solvers used in our calculations with the closures. Finally in \Cref{sec:Numerical Results}, we present results using our trained closures.

\section{Moment Closure for the BGK Equation}
\label{sect:Moment Closure for the BGK Equation}

In this section, we review the moment expansion and closure problem for the Bhatnagar-Gross-Krook (BGK) model and develop the necessary structure for a neural network closure.  

The BGK model \cite{Bhatnagar1954} is a simplification of the Boltzmann kinetic equation which deals specifically with the collision term. The BGK model assumes the distribution, $f$, tends to a local Maxwellian distribution and has several collision invariants which recover conservation laws \cite{Struchtrup2005}. The BGK equation is given by
\begin{equation}
    \label{eq:BGK}
    \frac{\partial f}{\partial t} + v \frac{\partial f}{\partial x} = \frac{1}{\tau} \left( f_M -f\right),
\end{equation}
where the unknown function $f=f(x,v,t)$ is the distribution (or called the phase density), $v$ is the microscopic velocity, $\tau>0$ is a relaxation time related to the Knudsen number, and $f_M$ is the Maxwellian distribution given by
\begin{equation}
    \label{eq:f_M Maxwellian}
    f_M = \frac{\rho}{\sqrt{2\pi\theta}}\exp\left(-\frac{(v-u)^2}{{2\theta}}\right).
\end{equation}
Physical density, temperature, and macroscopic velocity are given by $\rho$, $\theta$, and $u$ respectively \cite{Cai2013}. 

The BGK model has several collision invariant quantities. Let $\mathbf{\Phi}(v) = \left( 1, v, \frac{1}{2}v^2 \right)$, we have the following relationships \cite{Cai2013,Struchtrup2005}:
\begin{equation}
    \label{eq:primitive variable moments}
    \int f\, \mathbf{\Phi}(v) \,dv = 
    \begin{pmatrix}
    \rho \\ \rho u \\ \frac{1}{2}\rho \theta + \frac{1}{2}\rho u^2
    \end{pmatrix}.
\end{equation}
Applying $\mathbf{\Phi}$ to both sides of the \cref{eq:BGK} and integrating over velocity, we recover the conservation laws:
\begin{align}
\label{eq:conservation-mass}
\frac{\partial\rho}{\partial t} + \frac{\partial}{\partial x} \left(\rho u \right) &= 0, \\
\label{eq:conservation-momentum}
\frac{\partial}{\partial t}\left(\rho u\right) +  \frac{\partial}{\partial x} \left(\rho\theta + \rho u^2 \right) &= 0, \\
\label{eq:conservation-energy}
\frac{\partial}{\partial t}\left( \frac{1}{2}\rho\theta +\frac{1}{2}\rho u^2\right) +\frac{\partial}{\partial x}\left( q + \frac{3}{2}\rho u \theta + \frac{1}{2}\rho u^3\right) &=0,
\end{align}
where $q$ is the heat flux.

\subsection{Moment Method}
A common method to solve for the distribution of a kinetic model is the so-called moment method, where the distribution is expanded by a finite number of moments. Due to the finite expansion and relationship of the moments, the system of equations developed by the moment expansion is not closed and the system requires a closure in the highest-order moment.

In the seminal paper \cite{Grad1949}, Grad proposed a moment expansion using Hermite polynomials in the following manner
\begin{equation}
\label{eq:grad_moment_expansion}
f(x,v,t) = \sum^\infty_{k=0} \frac{1}{\sqrt{2\pi}}\theta^{-\frac{k+1}{2}}\He_k(z)\exp{\left(-\frac{z^2}{2}\right)}f_k(x,t),
\end{equation}
where $z := \frac{v-u}{\sqrt{\theta}}$, $\He_k$ is the $k$-th order Hermite polynomial, and $f_k=f_k(x,t)$ is called the $k$-th order moment. It is easy to show that the first four moments in the expansion \eqref{eq:grad_moment_expansion} satisfy $f_0 = \rho$, $f_1 =f_2=0$, and $3f_3=q$. Continuing for $k>3$, the evolution of the moment can be derived using the properties of Hermite polynomials \cite{Grad1949}:
\begin{equation}
\label{eq:grad moment recursion relation}
\begin{split}
&\frac{\partial f_k}{\partial t} - \frac{1}{\rho}f_{k-2}\frac{\partial q}{\partial x} -f_{k-1}\frac{\theta}{\rho} \frac{\partial \rho}{\partial x} + \frac{\partial \theta}{\partial x}\left( \frac{k-1}{2}f_{k-1} +\frac{\theta}{2} f_{k-3} \right) + (k+1)f_k\frac{\partial u}{\partial x} \\
&\quad\quad\quad\quad\quad\quad\quad\quad + u\frac{\partial f_k}{\partial x} + \theta \frac{\partial f_{k-1}}{\partial x} + (k+1)\frac{\partial f_{k+1}}{\partial x}=-\frac{1}{\tau} f_k.
\end{split}
\end{equation}
This relation reveals the closure problem since the time evolution of $f_k$ has a dependence on the next order moment $f_{k+1}$.

Truncating the moment hierarchy up to order $M$ and setting $f_{M+1} = 0$ yields the so-called Grad's moment closure system.
Formulating the moment system from \cref{eq:conservation-energy,eq:conservation-mass,eq:conservation-momentum,eq:grad moment recursion relation} in primitive variables $\mathbf{w} = \left(\rho,u,\theta,f_3,f_4,...,f_M\right)^T$, we arrive at the Grad's moment system in non-conservative form \cite{Grad1949,Cai2013}
\begin{equation}
    \label{eq:primitive moment system matrix vector}
    \frac{\partial{\mathbf{w}}}{\partial t} + A^{\text{Grad}}\frac{\partial{\mathbf{w}}}{\partial x} = \mathbf{Q},
\end{equation}
where $A^{\text{Grad}}$ is a  matrix of the form
 \begin{equation}
    \small
    \label{eq:Grad's moment matrix}
    \begin{pmatrix}
        u                           & \rho     & 0                                            & 0                      & 0      & 0      & 0      && \cdots & 0 \\
        \frac{\theta}{\rho}         & u        & 1                                            & 0                      & 0      & 0      & 0      && \cdots & 0 \\
        0                           & 2\theta  & u                                            & \frac{6}{\rho}         & 0      & 0      & 0      && \cdots & 0 \\
        0                           & 4f_3     & \frac{1}{2}\rho\theta                        & u                      & 4      & 0      & 0      && \cdots & 0 \\
        -\frac{\theta}{\rho}f_3     & 5f_4     & \frac{3}{2}f_3                               & \theta                 & u      & 5      & 0      && \cdots & 0 \\
        -\frac{\theta}{\rho}f_4     & 6f_5     & 2f_4                                         & -\frac{3}{\rho}f_3     & \theta & u      & 6      && \cdots & 0 \\
        -\frac{\theta}{\rho}f_5     & 7f_6     & \frac{\theta}{2}f_3+\frac{5}{2}f_5           & -\frac{3}{\rho}f_4     & 0      & \theta & u      && \cdots & 0 \\
        \vdots                      & \vdots   & \vdots                                       & \vdots                 & \vdots && \ddots & \ddots & \ddots & 0 \\
        -\frac{\theta}{\rho}f_{M-2} & Mf_{M-1} & \frac{\theta}{2}f_{M-4}+\frac{M-2}{2}f_{M-2} & -\frac{3}{\rho}f_{M-3} & 0      &\cdots  &0& \theta  & u      & M \\
        -\frac{\theta}{\rho}f_{M-1} & (M+1)f_M & \frac{\theta}{2}f_{M-3}+\frac{M-1}{2}f_{M-1} & -\frac{3}{\rho}f_{M-2} & 0      & \cdots   & & 0& \theta & u
    \end{pmatrix}.
\end{equation}
and $\textbf{Q}$ is the collision term
\begin{equation*}
    \textbf{Q} = (0, 0, 0, -\frac{1}{\tau}f_3, -\frac{1}{\tau}f_4, \cdots, -\frac{1}{\tau}f_M)^T.
\end{equation*}

However, the Grad's moment system \eqref{eq:primitive moment system matrix vector} is not globally hyperbolic, i.e. $A^{\text{Grad}}$ is not generally real diagonalizable. To fix this problem, Cai \textit{et al.}\ utilize independence between moment expansion coefficients and the Jacobian's characteristic polynomial from Grad's moment system, leading to a regularization that gives global hyperbolicity \cite{Cai2013}. This is the so-called Hyperbolic Moment Equations (HME).
The change in HME occurs in the last row of matrix \cref{eq:Grad's moment matrix}, for the second and third entries, which now becomes
\begin{equation}
    \begin{pmatrix}
         -\frac{\theta}{\rho}f_{M-1} & 0 & -f_{M-1} +\frac{\theta}{2}f_{M-3} & -\frac{3}{\rho}f_{M-2} & 0 & \cdots & 0& \theta & u
    \end{pmatrix}.
\end{equation}

In this paper, our goal is to replace the last row of this matrix with a relationship trained by deep learning methods in a way that captures the kinetic effects of the BGK equation but with a relatively small number of moments.

\subsection{Deep Learning Moment Closure}
We approach the closure problem following our recent work on the radiative transfer equation (RTE) \cite{Huang2022-RTE1}. This work differs from the approaches in \cite{Han-WeinanE-2019,Li2023} by learning the gradients of the moments instead of the moments directly, and in follow-up efforts \cite{Huang-RTE2,Huang-RTE3} preserving hyperbolicity of the system.

\subsubsection{Learning the Gradient of Moments and Hyperbolicity}
As shown in the RTE work \cite{Huang2022-RTE1}, in the free streaming limit, where the scattering and absorption cross sections vanish, with isotropic initial data, an exact moment closure can be derived. Motivated by this, Huang \textit{et al.}\ proposed to directly learn the gradient of the unclosed moment \cite{Huang2022-RTE1}.
Generalizing the same idea to the BGK model, we seek a closure relation
\begin{equation}
    \label{eq:NN closure}
    \frac{\partial f_{M+1}}{\partial x} = \sum_{i=0}^{M} \mathcal{N}_i(\rho,u,\theta,f_3,...,f_M)\, \frac{\partial f_i}{\partial x},
\end{equation}
where $\mathcal{N}_i$ is the $i$-th output from the neural network and is a nonlinear function depending on the lower-order moments, $(\rho,u,\theta,f_3,...,f_M)$. Substituting \cref{eq:NN closure} into \cref{eq:grad moment recursion relation} we have
\begin{equation}
    \label{eq:NN closure in recursive formula}
    \begin{split}
    &\frac{\partial f_M}{\partial t} - \frac{1}{\rho}f_{M-2}\frac{\partial q}{\partial x} -f_{M-1}\frac{\theta}{\rho} \frac{\partial \rho}{\partial x} + \frac{\partial \theta}{\partial x}\left( \frac{M-1}{2}f_{M-1} +\frac{\theta}{2} f_{M-3} \right) + (M+1)f_M\frac{\partial u}{\partial x} \nonumber\\ 
    &\quad\quad + u\frac{\partial f_M}{\partial x} +\theta \frac{\partial f_{M-1}}{\partial x}+(M+1)\sum_{i=0}^{M} \mathcal{N}_i(\rho,u,\theta,f_3,...,f_M)\, \frac{\partial f_i}{\partial x}=0.
\end{split}
\end{equation}
Thus the moment system \cref{eq:primitive moment system matrix vector} is now changed with the last row of matrix \cref{eq:Grad's moment matrix} given by
\begin{equation}
\label{eq:moment system with ML closure}
\begin{pmatrix}
    a_0 & a_1 & a_2 & a_3 & \cdots & a_{M-1} & a_{M}
\end{pmatrix},
\end{equation}
where
\begin{equation}
    \label{eq:NN coefficients for closure}
    a_i = 
    \begin{cases}
        -\frac{\theta}{\rho}f_{M-1} + (M+1)\mathcal{N}_0, & i=0, \\
        (M+1)f_M + (M+1)\mathcal{N}_1, & i=1, \\
        \frac{\theta}{\rho}f_{M-3} + \frac{M-1}{2}f_{M-1} + (M+1)\mathcal{N}_2, & i=2, \\
        -\frac{3}{\rho}f_{M-2} + (M+1)\mathcal{N}_3, & i=3, \\
        \theta + (M+1)\mathcal{N}_{M-1}, & i= M -1, \\
        u+ (M+1)\mathcal{N}_M, & i=M,\\
        (M+1)\mathcal{N}_i, & \text{otherwise}.
    \end{cases}
\end{equation}
We denote this ML moment closure system by
\begin{equation}
    \label{eq:ML moment system}
    \frac{\partial{\mathbf{w}}}{\partial t} + A^{\text{ML}}\frac{\partial{\mathbf{w}}}{\partial x} = \textbf{Q},
\end{equation}
where $A^{\text{ML}}$ is a matrix which only differs from $A^{\text{Grad}}$ in the last row. Next, we will derive some constraints on the neural networks $\mathcal{N}_i$ for $0\le i\le M$ in \eqref{eq:NN closure} such that the system \eqref{eq:ML moment system} satisfies Galilean invariance and hyperbolicity.

\subsubsection{Galilean Invariance}
Our moment closure should preserve Galilean invariance. We prove a condition for Galilean invariance.
\begin{thm}\label{theorem:Galilean-invariance}
Consider the following moment closure system without collision terms
    \begin{equation}
        \label{eq:invariance pf. moment system}
        \frac{\partial \mathbf{w}}{\partial t} + A(\mathbf{w})\frac{\partial \mathbf{w}}{\partial x} = 0,
    \end{equation}
    where $\mathbf{w} = (\rho, u, \theta, f_3, f_4,\cdots,f_M)^T \in \mathbb{R}^{M+1}$. The system is Galilean invariant if and only if $(A(\mathbf{w})-uI)$ is independent of $u$ where $I$ is the identity matrix.    
\end{thm}

\begin{proof}
Consider two inertial (non-accelerating) frames of reference given by $(x,t)$ and $(x^*,t^*)$. The frames are equivalent at an initial time and the second frame is moving at a constant velocity, $c$, with respect to the other frame. That is we have $t^* = t$ and $x^* = x -ct$.

The Galilean invariance means that, if $w(x,t)$ satisfies \cref{eq:invariance pf. moment system}, then $\mathbf{w}^*(x^*,t^*) = (\rho^*, u^*, \theta^*, f_3^*,\cdots,f_M^*)^T$ with $\rho = \rho^*$, $u^* = u-c$, $\theta^* = \theta$, and $f_k^* = f_k$ for all $k$ such that $3\leq k \leq M$, also satisfies \cref{eq:invariance pf. moment system}.

    Applying the chain rule we have
    \begin{equation*}
        \frac{\partial }{\partial t} = \frac{\partial t^*}{\partial t}\frac{\partial }{\partial t^*} + \frac{\partial x^*}{\partial t}\frac{\partial }{\partial x^*} =\frac{\partial }{\partial t^*} - c\frac{\partial }{\partial x^*}
    \end{equation*}
    and
    \begin{equation*}
        \frac{\partial }{\partial x} = \frac{\partial t^*}{\partial x}\frac{\partial }{\partial t^*} + \frac{\partial x^*}{\partial x}\frac{\partial }{\partial x^*} =\frac{\partial }{\partial x^*}.
    \end{equation*}

    Proving ($\impliedby$): Suppose $\mathbf{w}=\mathbf{w}(x,t)$ solves \cref{eq:invariance pf. moment system}. Then $\mathbf{w} = \mathbf{w}(x^*,t^*)$ solves
    \begin{equation*}
        \frac{\partial \mathbf{w}}{\partial t^*}-c\frac{\partial \mathbf{w}}{\partial x^*} + A(\mathbf{w})\frac{\partial \mathbf{w}}{\partial x^*} = 0
    \end{equation*}
    rearranging
    \begin{equation*}
        \frac{\partial \mathbf{w}}{\partial t^*}+ (A(\mathbf{w})-cI)\frac{\partial \mathbf{w}}{\partial x^*} = 0
    \end{equation*}
    Since $A(\mathbf{w}) -uI$ is independent of $u$
    \begin{equation*}
        A(\mathbf{w})-uI = A(\mathbf{w}^*)-u^*I = A(\mathbf{w}^*) - (u-c)I
    \end{equation*}
    or
    \begin{equation*}
        A(\mathbf{w}^*) = A(\mathbf{w})-cI
    \end{equation*}
    Thus $\mathbf{w}^*(x^*,t^*)$ satisfies $\frac{\partial \mathbf{w}^*}{\partial t^*} + A(\mathbf{w}^*)\frac{\partial \mathbf{w}^*}{\partial x^*} = 0$.
    
    Proving ($\implies$): Comparing equations involving $\mathbf{w}$ and $\mathbf{w}^*$ we have
    \begin{equation*}
        A(\mathbf{w}^*) = A(\mathbf{w})-cI
    \end{equation*}
    Differentiating with respect to $c$,
    \begin{equation*}
        -\frac{\partial A(\mathbf{w}^*)}{\partial c} = -I
    \end{equation*}
    and thus 
    \begin{equation*}
        \frac{\partial}{\partial u}(A(\mathbf{w}^*)-uI) = 0.
    \end{equation*}
\end{proof}

From Theorem \ref{theorem:Galilean-invariance}, it is easy to show the following corollary:
\begin{cor}
Consider the ML moment closure system \eqref{eq:ML moment system}. Denote the eigenvalues of $A^{\text{ML}}$ by $r_k$ for $0\leq k \leq M$. The system is Galilean invariant if and only if any of the two conditions is satisfied:
    \begin{enumerate}
        \item $a_k$ for $0\leq k \leq M-1$ and $(a_m-u)$ are independent of $u$.
        \item $(r_k -u)$ for $0\leq k\leq M$ are independent of $u$.
    \end{enumerate}
\end{cor}

From the above corollary, to guarantee the Galilean invariance, the eigenvalues of $A^{\text{ML}}$ in \eqref{eq:ML moment system} must have the form $r_k = u + \tilde{r}_k$ where $\tilde{r}_k$ is independent of $u$.

\subsubsection{Provable Hyperbolic Structure}

By taking advantage of the moment system structure, we will develop the framework for the neural network to maintain hyperbolicity by relating the coefficients of gradients to the eigenvalues of the moment system. First, we recall some definitions and theorems in \cite{Huang-RTE3}.

\begin{definition}[lower Hessenberg matrix \cite{Huang-RTE3}] 
\label{def: lower hessenberg}
The matrix $H=(h_{ij})_{n\times n}$ is called lower Hessenberg matrix if $h_{ij}=0$ for $j>i+1$. If $h_{i,i+1} \neq 0$ for $i=1,2,...,n-1$, it is called unreduced lower Hessenberg matrix.
\end{definition}

\begin{definition}[associated polynomial sequence \cite{Elouafi2009, Huang-RTE3}]
\label{def:assoc polynomial seq.}
Let $H=(h_{ij})_{n\times n}$ be an unreduced lower Hessenberg matrix. The associated polynomial sequence $\{ q_i \}_{0\leq i\leq n}$ with $H$ is defined as $q_0=1$ and
\begin{equation}
    \label{eq:assoc polynomial recurrence}
    q_i(x) = \frac{1}{h_{i,i+1}}\left( xq_{i-1}(x) - \sum_{j=1}^i h_{ij}q_{j-1}(x) \right),\quad 1\leq i\leq n,  
\end{equation}
with $h_{n,n+1}=1.$
\end{definition}

\begin{theorem}[\cite{Elouafi2009, Huang-RTE3}]
    \label{thm:eigenvalue vectors for assoc. polynomial}
    Let $H=(h_{ij})_{n\times n}$ be an unreduced lower Hessenberg matrix and $\{ q_i \}_{0\leq i\leq n}$ is the associated polynomials sequence with $H$. The following are true:
    \begin{enumerate}
        \item If $\lambda$ is a root of $q_n$, then $\lambda$ is an eigenvalue of $H$ and a corresponding eigenvector is $(q_0(\lambda),q_1(\lambda),...,q_{n-1}(\lambda))^T$.
        \item If all the roots of $q_n$ are simple, then the characteristic polynomial of $H$ is $\mu q_n$ where $\mu = \prod_{i=1}^{n-1}h_{i,i+1}$, that is
        \begin{equation}
            \det(x I -H)  = \mu q_n(x),
        \end{equation}
        where $I$ is the identity matrix.
    \end{enumerate}
\end{theorem}

\begin{theorem}[\cite{Huang-RTE3}]
    \label{thm:consequences of hessenberg, associated polynomial, eigens}
    Let $H=(h_{ij})_{n\times n}$ be an unreduced lower Hessenberg matrix and $\{ q_i \}_{0\leq i\leq n}$ is the associated polynomials sequence with $H$. The following are equivalent:
    \begin{enumerate}
        \item $H$ is real diagonalizable,
        \item all eigenvalues of $H$ are distinct and real,
        \item all roots of $q_n$ are simple and real.
    \end{enumerate}
\end{theorem}

The proceeding theorems provide a method to ensure hyperbolicity of the moment system when training a neural network based closure. We derive the associated polynomial sequence for the matrix of the ML moment system \cref{eq:ML moment system}.

\begin{theorem}
    The associated polynomial sequence $\{q_k(x)\}$ where ${0\leq k \leq M+1}$ for $A^{\text{ML}}$ in \eqref{eq:ML moment system} satisfies the following relationship:
    \begin{align}
        q_0(x) &= \He_0, \quad q_1(x) = \frac{\theta^{\frac{1}{2}}}{\rho}\He_1, \quad q_2(x) = \frac{\theta}{\rho}\He_2, \quad q_3(x) = \frac{\theta^\frac{3}{2}}{6}\He_3, \nonumber \\
        q_k(x) &= \frac{\theta^\frac{k}{2}}{k!}\He_k - \frac{\theta f_{k-2}}{2\rho}\He_2 - \frac{\theta^\frac{1}{2} f_{k-1}}{\rho}\He_1,\quad 4\leq k \leq M, \nonumber \\
        q_{M+1}(x) &= \frac{\theta^{\frac{M+1}{2}}}{M!}\He_{M+1} - (a_M - u)\frac{\theta^\frac{M}{2}}{M!}\He_M \nonumber \\
        &\quad +\left( \frac{\theta^\frac{M+1}{2}}{(M-1)!} -a_{M-1}\frac{\theta^\frac{M-1}{2}}{(M-1)!} \right)\He_{M-1} \nonumber\\
        &\quad -\sum^{M-2}_{k=4} a_k \frac{\theta^{\frac{k}{2}}}{k!}\He_k + \left( -\frac{\theta^\frac{3}{2}f_{M-2}}{2\rho} -a_3\frac{\theta^\frac{3}{2}}{6} \right)\He_3 \nonumber \\
        &\quad +\left(-\frac{\theta f_{M-1}}{\rho} + (a_M-u)\frac{\theta f_{M-2}}{2\rho} + \sum^M_{k=6}a_{k-1}\frac{\theta f_{k-3}}{2\rho} -a_2\frac{\theta}{\rho} \right)\He_2 \nonumber \\
        &\quad + \left(-\frac{\theta^\frac{3}{2} f_{M-2}}{\rho} + (a_M -u)\frac{\theta^\frac{1}{2}f_{M-1}}{\rho} + \sum^M_{k=5}a_{k-1}\frac{\theta^\frac{1}{2}f_{k-2}}{\rho} -a_1\frac{\theta^\frac{1}{2}}{\rho}\right)\He_1 \nonumber \\
        &\quad + \left( -\frac{\theta f_{M-1}}{\rho} -a_0\right)\He_0. \nonumber
    \end{align}
    where the argument $\xi = \frac{x-u}{\sqrt{\theta}}$ in the Hermite polynomials is omitted. {  If further assuming all the roots of $q_{M+1}$ are simple,} the eigenpolynomial of $A^{ML}$ satisfies:
    \begin{equation}
        \label{eq:eigen-polynomial-ml-2}
        p_{M+1}(x) = M!q_{M+1}(x).
    \end{equation}
\end{theorem}

\begin{proof}
Using Definition \ref{def:assoc polynomial seq.}, $q_i(x)$ for $0\le i\le 3$ can be directly calculated, and the details are omitted here.

Next, we assume that the relationship holds for $k=n-1$ with $4\le n\le M$ and we would like to prove that it also holds for $k=n$. By definition \eqref{eq:assoc polynomial recurrence} we have
\begin{align}
    nq_n(x) ={}& (x-u)q_{n-1}(x) + \frac{\theta f_{n-2}}{\rho}q_0(x) -nf_{n-1}q_1(x) \nonumber \\
    & - \frac{1}{2}((n-2)f_{n-2} +\theta f_{n-4})q_2(x) + \frac{3f_{n-3}}{\rho}q_3(x) -\theta q_{n-2} \nonumber \\
    ={}& (x-u)\left( \frac{\theta^\frac{n-1}{2}}{(n-1)!}\He_{n-1}(\xi) - \frac{\theta f_{n-3}}{2\rho}\He_2(\xi) - \frac{\theta^\frac{1}{2} f_{k-2}}{\rho}\He_1(\xi) \right) \nonumber \\
    & + \frac{\theta f_{n-2}}{\rho}\He_0(\xi) -nf_{n-1}\frac{\theta^\frac{1}{2}}{\rho}\He_1(\xi) - \frac{1}{2}((n-2)f_{n-2} +\theta f_{n-4})\frac{\theta}{\rho}\He_2(\xi)\nonumber \\
    & + \frac{3f_{n-3}}{\rho}\frac{\theta^\frac{3}{2}}{6}\He_3(\xi) \nonumber \\ 
    & -\theta \left( \frac{\theta^\frac{n-2}{2}}{(n-2)!}\He_{n-2}(\xi) - \frac{\theta f_{n-4}}{2\rho}\He_2(\xi) - \frac{\theta^\frac{1}{2} f_{k-3}}{\rho}\He_1(\xi) \right) \nonumber \\
    ={}& \theta^{\frac{1}{2}}\xi\left( \frac{\theta^\frac{n-1}{2}}{(n-1)!}\He_{n-1}(\xi) - \frac{\theta f_{n-3}}{2\rho}\He_2(\xi) - \frac{\theta^\frac{1}{2} f_{k-2}}{\rho}\He_1(\xi) \right) \nonumber \\
    & + \frac{\theta f_{n-2}}{\rho}\He_0(\xi) -nf_{n-1}\frac{\theta^\frac{1}{2}}{\rho}\He_1(\xi) - \frac{1}{2}((n-2)f_{n-2} +\theta f_{n-4})\frac{\theta}{\rho}\He_2(\xi)\nonumber \\
    & + \frac{3f_{n-3}}{\rho}\frac{\theta^\frac{3}{2}}{6}\He_3(\xi) \nonumber \\
    & -\theta \left( \frac{\theta^\frac{n-2}{2}}{(n-2)!}\He_{n-2}(\xi) - \frac{\theta f_{n-4}}{2\rho}\He_2(\xi) - \frac{\theta^\frac{1}{2} f_{n-3}}{\rho}\He_1(\xi) \right) \nonumber \\
    ={}& \frac{\theta^\frac{n}{2}}{(n-1)!}\He_n(\xi) - n\frac{\theta f_{n-2}}{2\rho}\He_2(\xi) -n \frac{\theta^\frac{1}{2} f_{n-1}}{\rho}\He_1(\xi), \nonumber
\end{align}
where we use $x-u = \theta^{\frac{1}{2}} \xi$ and the recurrence relation $\xi\He_k(\xi) = \He_{k+1}(\xi) + k \He_{k-1}(\xi)$. Thus we have
\begin{equation*}
    q_n(x) = \frac{\theta^\frac{n}{2}}{n!}\He_n(\xi) - \frac{\theta f_{n-2}}{2\rho}\He_2(\xi) - \frac{\theta^\frac{1}{2} f_{n-1}}{\rho}\He_1(\xi). \nonumber
\end{equation*}

For $k = M+1$, we have
    \begin{align*}
        q_{M+1}(x) ={}& (x-a_M)\left( \frac{\theta^\frac{M}{2}}{M!}\He_M -\frac{\theta f_{M-2}}{2\rho}\He_2 - \frac{\theta^\frac{1}{2}f_{M-1}}{\rho}\He_1 \right) \nonumber \\
        & -a_{M-1}\left( \frac{\theta^\frac{M-1}{2}}{(M-1)!}\He_{M-1} -\frac{\theta f_{M-3}}{2\rho}\He_2 - \frac{\theta^\frac{1}{2}f_{M-2}}{\rho}\He_1 \right) \nonumber \\
        & \cdots \nonumber \\
        & -a_5\left( \frac{\theta^\frac{5}{2}}{5!}\He_5 - \frac{\theta f_3}{2\rho}\He_2 - \frac{\theta^\frac{1}{2}f_4}{\rho}\He_1 \right) \nonumber \\
        & -a_4\left( \frac{\theta^\frac{4}{2}}{4!}\He_4 - \frac{\theta^\frac{1}{2}f_3}{\rho}\He_1 \right) \nonumber \\
        & -a_3\frac{\theta^\frac{3}{2}}{6}\He_3 -a_2\frac{\theta}{\rho}\He_2 -a_1\frac{\theta^\frac{1}{2}}{\rho}\He_1 -a_0\He_0,
    \end{align*}
    where we omit the argument in the Hermite polynomials.     
    The first term can be simplified using the recurrence relation:
    \begin{align}
        & (x-a_M)\left( \frac{\theta^\frac{M}{2}}{M!}\He_M -\frac{\theta f_{M-2}}{2\rho}\He_2 - \frac{\theta^\frac{1}{2}f_{M-1}}{\rho}\He_1 \right) \nonumber \\
        ={}& (x-u)\left( \frac{\theta^\frac{M}{2}}{M!}\He_M -\frac{\theta f_{M-2}}{2\rho}\He_2 - \frac{\theta^\frac{1}{2}f_{M-1}}{\rho}\He_1 \right) \nonumber \\
        & - (a_M - u)\left( \frac{\theta^\frac{M}{2}}{M!}\He_M -\frac{\theta f_{M-2}}{2\rho}\He_2 - \frac{\theta^\frac{1}{2}f_{M-1}}{\rho}\He_1 \right) \nonumber \\
        ={}& \frac{\theta^\frac{M}{2}}{M!}(\theta^\frac{1}{2}\He_{M+1} + M\theta^\frac{1}{2}\He_{M-1}) - \frac{\theta f_{M-2}}{2\rho}(\theta^\frac{1}{2}\He_3 +2\theta^\frac{1}{2}\He_1) \nonumber \\
        & -\frac{\theta^\frac{1}{2}f_{M-1}}{\rho}(\theta^\frac{1}{2}\He_2 +\theta^\frac{1}{2}\He_0) \nonumber \\
        & - (a_M -u)\left( \frac{\theta^\frac{M}{2}}{M!}\He_M -\frac{\theta f_{M-2}}{2\rho}\He_2  - \frac{\theta^\frac{1}{2}f_{M-1}}{\rho}\He_1 \right) \nonumber.
    \end{align}
    Then,
    \begin{align*}
        q_{M+1}(x) &= \frac{\theta^\frac{M}{2}}{M!}(\theta^\frac{1}{2}\He_{M+1} + M\theta^\frac{1}{2}\He_{M-1}) - \frac{\theta f_{M-2}}{2\rho}(\theta^\frac{1}{2}\He_3 +2\theta^\frac{1}{2}\He_1) \nonumber \\
        &\quad -\frac{\theta^\frac{1}{2}f_{M-1}}{\rho}(\theta^\frac{1}{2}\He_2 +\theta^\frac{1}{2}\He_0) \nonumber \\
        &\quad - (a_M -u)\left( \frac{\theta^\frac{M}{2}}{M!}\He_M -\frac{\theta f_{M-2}}{2\rho}\He_2  - \frac{\theta^\frac{1}{2}f_{M-1}}{\rho}\He_1 \right) \nonumber \\
        &\quad -a_{M-1}\left( \frac{\theta^\frac{M-1}{2}}{(M-1)!}\He_{M-1} -\frac{\theta f_{M-3}}{2\rho}\He_2 - \frac{\theta^\frac{1}{2}f_{M-2}}{\rho}\He_1 \right) \nonumber \\
        &\quad \cdots \nonumber \\
        &\quad -a_5\left( \frac{\theta^\frac{5}{2}}{5!}\He_5 - \frac{\theta f_3}{2\rho}\He_2 - \frac{\theta^\frac{1}{2}f_4}{\rho}\He_1 \right) \nonumber \\
        &\quad -a_4\left( \frac{\theta^\frac{4}{2}}{4!}\He_4 - \frac{\theta^\frac{1}{2}f_3}{\rho}\He_1 \right) \nonumber \\
        &\quad -a_3\frac{\theta^\frac{3}{2}}{6}\He_3 -a_2\frac{\theta}{\rho}\He_2 -a_1\frac{\theta^\frac{1}{2}}{\rho}\He_1 -a_0\He_0 \nonumber \\
        &= \frac{\theta^{\frac{M+1}{2}}}{M!}\He_{M+1} - (a_M - u)\frac{\theta^\frac{M}{2}}{M!}\He_M \nonumber \\
        &\quad +\left( \frac{\theta^\frac{M+1}{2}}{(M-1)!} -a_{M-1}\frac{\theta^\frac{M-1}{2}}{(M-1)!} \right)\He_{M-1} \nonumber\\
        &\quad -\sum^{M-2}_{k=4} a_k \frac{\theta^{\frac{k}{2}}}{k!}\He_k + \left( -\frac{\theta^\frac{3}{2}f_{M-2}}{2\rho} -a_3\frac{\theta^\frac{3}{2}}{6} \right)\He_3 \nonumber \\
        &\quad +\left(-\frac{\theta f_{M-1}}{\rho} + (a_M-u)\frac{\theta f_{M-2}}{2\rho} + \sum^M_{k=6}a_{k-1}\frac{\theta f_{k-3}}{2\rho} -a_2\frac{\theta}{\rho} \right)\He_2 \nonumber \\
        &\quad + \left(-\frac{\theta^\frac{3}{2} f_{M-2}}{\rho} + (a_M -u)\frac{\theta^\frac{1}{2}f_{M-1}}{\rho} + \sum^M_{k=5}a_{k-1}\frac{\theta^\frac{1}{2}f_{k-2}}{\rho} -a_1\frac{\theta^\frac{1}{2}}{\rho}\right)\He_1 \nonumber \\
        &\quad + \left( -\frac{\theta f_{M-1}}{\rho} -a_0\right)\He_0.
    \end{align*}

    From Theorem \ref{thm:eigenvalue vectors for assoc. polynomial} we see that the characteristic polynomial is given by $p_{M+1}(x) = M!q_{M+1}(x)$.
\end{proof}

\subsubsection{Translating eigenvalues to coefficients of gradients}
Proceeding as in the RTE effort \cite{Huang-RTE3}, we can derive a relationship between the eigenvalues and the weights of the gradient based closure using Vieta's formula as well as a transformation between monomial coefficients and the Hermite polynomials.

The relation between monomials and the Hermite polynomials is \cite{NIST:DLMF}
\begin{equation}
    x^m = \sum_{k=0}^{\lfloor \frac{m}{2}\rfloor} F(m,k) \He_{m-2k}(x)
\end{equation}
where $F(m,k) = \frac{m!}{2^k k! (m-2k)!}$. We can rewrite this into an equivalent form
\begin{equation}
    x^m = \sum_{k=0}^m b_{mk}\He_k(x),\quad m\geq 0,
\end{equation}
with
\begin{equation}
    b_{mk} = \begin{cases}
        F(m,\frac{1}{2}(m-k)), &\quad \text{if } m\equiv k \pmod 2 \\
        0, &\quad \text{otherwise.} 
    \end{cases}
\end{equation}
From this formula we can expand any polynomials in terms of Hermite polynomials.
\begin{equation*}
    \sum_{i=0}^n c_i x^i = \sum_{i=0}^n c_i\left(\sum_{k=0}^i b_{ik}\He_k(x) \right) = \sum_{i=0}^n \left(\sum_{i=k}^n c_i b_{ik}\right)\He_k(x) = \sum_{k=0}^n \alpha_k \He_k(x),
\end{equation*}
where $c_i$ are polynomials coefficients and the coefficients of Hermite polynomials are determined by 
\begin{equation}\label{eq:transform-monomial-to-hermite}
    \alpha_k = \sum_{i=k}^n c_i b_{ik}.
\end{equation}

Now we proceed to derive a relationship between the eigenvalues of $A^{\text{ML}}$ in \cref{eq:ML moment system} and the last row of learned coefficients. Let $r_k = u+ \tilde{r}_k$ for $k=0,1,...,M$ and let us write the eigenpolynomial as
\begin{equation}
    p_{M+1} = c_0 +c_1(x-u) + \cdots +c_M(x-u)^M + (x-u)^{M+1}
\end{equation}
then we have that
\begin{multline}
    (x-u-\tilde{r}_0)(x-u-\tilde{r}_1)\cdots(x-u-\tilde{r}_M) \\ = c_0 +c_1(x-u) + \cdots +c_M(x-u)^M + (x-u)^{M+1}
\end{multline}
Applying Vieta's formula, we can relate the coefficients $c_k$ to sums and products of $\tilde{r}_k$ in the following manner
\begin{align*}
    \tilde{r}_0 + \tilde{r}_1 + \cdots \tilde{r}_{M-1}+\tilde{r}_M &= -c_M \nonumber \\
    (\tilde{r}_0\tilde{r}_1 + \tilde{r}_0\tilde{r}_2 +\cdots \tilde{r}_0\tilde{r}_{M}) + (\tilde{r}_1\tilde{r}_2 + \tilde{r}_1\tilde{r}_3 +\cdots +\tilde{r}_1\tilde{r}_M) +\cdots + \tilde{r}_{M-1}\tilde{r}_M &= c_{M-1} \nonumber \\
    &\quad \vdots \nonumber \\
    \tilde{r}_0\tilde{r}_1\cdots\tilde{r}_{M-1}\tilde{r}_M &= (-1)^{M+1}c_0
\end{align*}
or equivalently
\begin{equation}
    \sum_{0\leq i_1 < i_2 <\cdots <i_k \leq M}\left(\prod_{j=1}^k \tilde{r}_{i_j} = (-1)^k c_{M+1-k}\right),
\end{equation}
where the indices $i_k$ are sorted strictly increasing to ensure each product of $k$ roots is used once.

We then transform the basis $(x-u)^k$ into the Hermite basis $\He_k\left(\frac{x-u}{\theta^\frac{1}{2}}\right)$,
\begin{align*}
    &c_0 + c_1(x-u) + \cdots +c_M(x-u)^M +(x-u)^{M+1} = \beta_0\He_0\left( \frac{x-u}{\theta^\frac{1}{2}}\right)\nonumber \\
     &\qquad + \beta_1 \theta^\frac{1}{2}\He_1\left( \frac{x-u}{\theta^\frac{1}{2}}\right) + \cdots + \beta_M \theta^\frac{M}{2}\He_M\left( \frac{x-u}{\theta^\frac{1}{2}}\right) + \theta^\frac{M+1}{2}\He_{M+1}\left( \frac{x-u}{\theta^\frac{1}{2}}\right).
\end{align*}
Let $s = \frac{x-u}{\theta^\frac{1}{2}}$ then we have
\begin{align*}
    c_0 + c_1\theta^\frac{1}{2}s &+ \cdots + c_M\theta^\frac{M}{2}s^M + \theta^\frac{M+1}{2}s^{M+1} \nonumber \\
    &=\beta_0\He_0(s) +\beta_1 \theta^\frac{1}{2} \He_1 (s) + \cdots + \beta_M \theta^\frac{M}{2}\He_M(s) + \theta^\frac{M+1}{2}\He_{M+1}(s)
\end{align*}
and thus by \eqref{eq:transform-monomial-to-hermite}
\begin{equation}
    \beta_k\theta^\frac{k}{2} = \sum_{i=k}^{M+1}c_i\theta^\frac{i}{2}b_{ik}, \quad k=0,...,M+1
\end{equation}
with $c_{M+1} = 1$. These $\beta_k$ coefficients can then be transformed to entries $a_k$ of \cref{eq:moment system with ML closure} using the equations of the eigenpolynomials for $A^{ML}$ \cref{eq:eigen-polynomial-ml-2}.

\section{Training Setup and Neural Network Architecture}
\label{sec:Training Setup and Neural Network Architecture}
Our effort seeks to apply the hyperbolicity preserving deep learning closure methods developed for RTE \cite{Huang-RTE3} and apply them to the BGK model. The goal is to capture kinetic effects in the learned closure by generating high-fidelity training data solving the BGK equation via discrete velocity methods (DVM) with smooth and discontinuous initial data \cite{Han-WeinanE-2019} over a range of Knudsen numbers. During the course of the investigation, several modifications were needed to the training methods and neural network architectures originally used in RTE \cite{Huang-RTE3} to improve training errors for the BGK model. This is likely due to the strong non-linearity of the BGK model compared to RTE. We discuss the machine learning aspects of our effort.
All deep learning methods were implemented using PyTorch \cite{NEURIPS2019_9015-Pytorch}.

{  In the numerical examples, we test two sets of ML moment closures. 
The first set is to recover the HME model using neural networks. To be more specific, we train the neural network using the data generated from the HME model. The purpose is to validate the correctness of the algorithm and code.
The second is to train the neural network using the data from the kinetic model. The goal is to capture the kinetic effects in the ML closure.}

\subsection{Training Data Preparation}
\label{subsec:training data preparation}
\FloatBarrier
We generate training data following \cite{Han-WeinanE-2019}. When generating data from the HME model, we use the high order path-conservative numerical scheme in section~\ref{sec:high-order-path-conser}. 
When generating the data from the kinetic model, we use the fifth-order finite difference WENO scheme in physical space \cite{jiang1996efficient}, discrete velocity method (DVM) in velocity space \cite{broadwell1964study}, and third-order implicit-explicit method in time \cite{Ascher1997}.
In \Cref{tab:Parameters-Training-Data-Combined}, we list the various parameters for the HME and kinetic model training data. 

Following \cite{Han-WeinanE-2019}, we have two sets of training data. The first set is generated from smooth data as the initial condition, which is also called wave data. First, we generate the following macroscopic variables
\begin{equation}
\label{eq:distribution generation}
    U = \begin{cases}
        \rho(x,0) = a_\rho \sin\left( \frac{2 k_\rho \pi x}{L} +\psi_\rho \right) + b_\rho, \\
        u(x,0) = 0, \\
        \theta(x,0) = a_\theta \sin\left( \frac{2 k_\theta \pi x}{L} +\psi_\theta \right) + b_\theta,
    \end{cases}
\end{equation}
where the random variables are uniformly sampled in their associated ranges $a_z\in [0.2,0.3]$, $\psi_z \in [0,2\pi]$, $b_z \in[0.5,0.7]$ and $k_z$ is a random integer sampled from the set $\{1,2,3,4 \}$ and $z$ can be either $\rho$ or $\theta$. 
Then two local Maxwellian distributions are generated from the above macroscopic variables $U_1$ and $U_2$ and then blended together to form the wave data:
\begin{equation}
    f_{wave} = \frac{\alpha_1 f_M(v; U_1) + \alpha_2 f_M(v; U_2) }{\alpha_1+\alpha_2 + \epsilon},
\end{equation}
where $f_M$ is the Maxwellian defined in \cref{eq:f_M Maxwellian}, and $\alpha_1$ and $\alpha_2$ are uniformly random variables in $[0,1]$, with  $\epsilon = 10^{-6}$ preventing division zero.

The second set of training data is called the mixed initial conditions, which are created by taking wave data $f_{wave}$ and superpositioning with shock data $f_{shock}$. The shock data is a Riemann problem where the interval $[x_1,x_2]$ has discontinuities at its endpoints, which are uniform random variables with $x_1\in [0.2,0.4]$ and $x_2 \in [0.6,0.8]$. The primitive variables are given on left side, $\rho_L,u_L,\theta_L$, and the right side, $\rho_R,u_L,\theta_R$, with $u_R=u_L =0$, $\rho_L$ and $\theta_L$ are two independent random variables uniformly sampled from $[1,2],$ and $\rho_R$ and $\theta_R$ are two independent random variables uniformly sampled from $[0.55,0.9]$. This shock condition is blended with smooth data as follows
\begin{equation}
    f_{mix} = \alpha f_{wave} + (1-\alpha) f_{shock},
\end{equation}
where $\alpha$ is a random variable uniformly sampled from $[0.2,0.4]$.

{  Following \cite{Han-WeinanE-2019}, we generate the training data to produce distribution functions with both smooth and discontinuous initial conditions and trajectories. These are a reasonable approximation to distributions of interest for the BGK equation with no external force. This method still limits the distributions seen in training by the frequency limits on the parameter used to generate the sinusoidal data in \eqref{eq:distribution generation}. Any model with external forces or associated fields, such as magneto-hydrodynamics, would need a different method to generate training data.}

\begin{table}[ht]
\centering
\begin{tabular}{ |c|c|c| }
\hline
\textbf{Parameters} & \textbf{HME} & \textbf{Kinetic}\\
\hline
Spatial Domain $x$: & $[-0.5,-0.5]$ & $[0,1]$\\
Spatial Mesh $N_x$: & 256 & 256\\
Time Domain $t$: & $[0,10]$ & $[0,1]$ \\
Time Steps $N_t$: & $320$ & $1000$\\
\multirow{2}{*}{Knudsen Number: }& Log-Uniform & Log-Uniform \\
& $[10^{-3}, 10]$& $[10^{-3}, 10]$\\
\hline
\end{tabular}
\caption{Parameters used to generate the training data for the HME model and kinetic model.}
\label{tab:Parameters-Training-Data-Combined}
\end{table}

\FloatBarrier
\subsection{Neural Network Architecture}
\label{sec:Neural Network Architecture}
The networks we utilize have the following structure, which is depicted in Figure \ref{fig:NN Architecture}. Moments from training data are input into a fully connected, neural network. The network has a repeating structure of block layers: linear layer, activation layer, and batch normalization layer. The output of the network, which is the eigenvalues of matrix \cref{eq:Grad's moment matrix}, with the last row as \cref{eq:moment system with ML closure}. This output is then passed to an implementation of Vieta's formula which is passed to a linear transform that matches the structure required in the last row of the matrix \cref{eq:NN coefficients for closure}. It is this final output that will be used in the non-conservative solver to close the system.

\begin{figure}[ht]
    \centering
    \includegraphics[width=.9\textwidth]{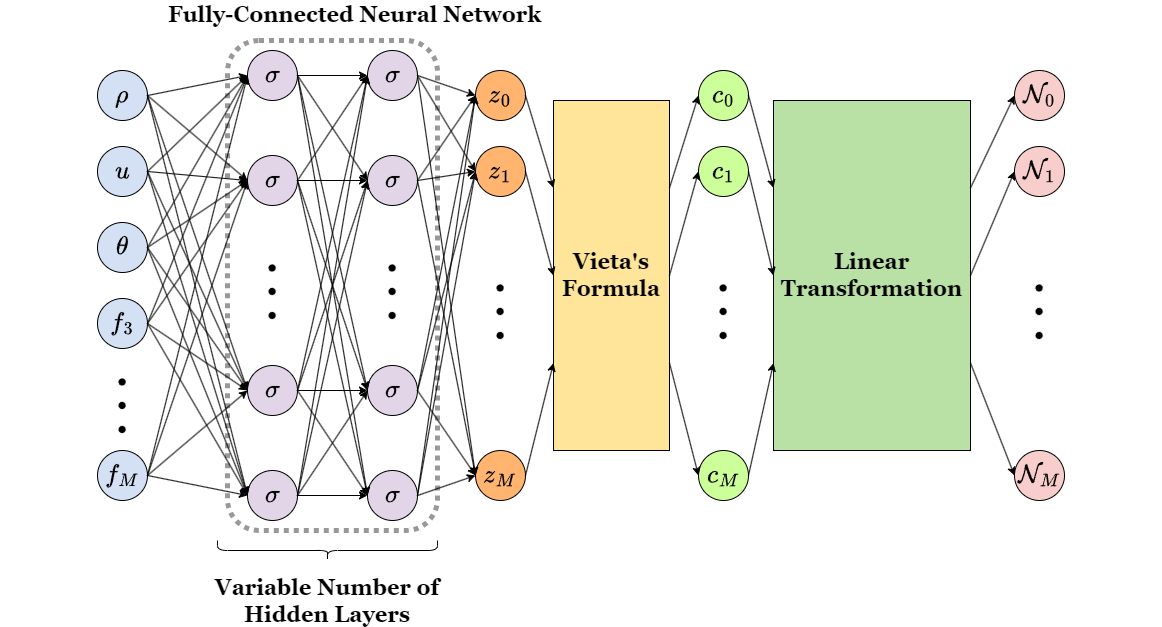}
    \caption{Neural network architecture from the left: the light blue $f_k$ are moments from the training data set. The boxed, gray circles with $\sigma$ represent the fully connected neural network, each column should be thought of as a block of layers (linear, activation, batch normalization). The yellow block applies Vieta's formula and the final green appropriately transforms the output to coefficients for the moment gradients. }
    \label{fig:NN Architecture}
\end{figure}

\FloatBarrier

\subsection{Training Considerations and Techniques}
\label{subsec:Training Considerations and Techniques}
For the training data recovering the HME closure, the training techniques and neural network structure followed that of previous RTE work \cite{Huang2022-RTE1,Huang-RTE2,Huang-RTE3}, with the modifications in Section \ref{sect:Moment Closure for the BGK Equation} for the BGK model. \Cref{tab:NN Parameters Combined} summarizes hyper-parameters for the HME closure training. 

For the kinetic training data, both the neural network architecture and training methods were modified to overcome difficulties minimizing training and testing error; specifically in the optimizer used to minimize the loss function, cycling the learning rate to avoid local minimums, and finally a change to the architecture with the inclusion of batch normalization layers to avoid vanishing and exploding gradients. We discuss each of these modifications individually. \Cref{tab:NN Parameters Combined} summarizes hyper-parameters the kinetic closure training. \Cref{tab:train test error for prediction} shows the associated training and testing error in the kinetic closure.

The loss function, $\mathcal{L}$, for training the neural network follows previous RTE work \cite{Huang2022-RTE1,Huang-RTE3}, and has the form,
\begin{equation}
\label{eq:loss-function-general-form}
    \mathcal{L} = \frac{1}{N_{data}}\sum_{j,n} | \partial_x f_{N+1}^{\text{Train}}(x_j,t_n)-\partial_x f_{N+1}^{\text{NN}}(x_j,t_n)|^2
\end{equation}
where term $\partial_x f_{N+1}^{\text{Train}}(x_j,t_n)$ is moment gradients from either the HME or kinetic training data and $\partial_x f_{N+1}^{\text{NN}} = \sum^N_{i = 0} \mathcal{N}_i(\rho, u, \theta, f_3, \cdots, f_N)\partial_x f_i$ is the neural network approximation on the associated training set, as in equation \cref{eq:NN closure}. In the HME model case, $\partial_x f_{N+1}^{\text{Train}}(x_j,t_n) = \partial_x f_{N+1}^{\text{HME}} = -f_N \partial_x u - \frac{f_{N-1}}{2} \partial_x \theta$ and is computed from the generated training data $(u, \theta, f_{N-1}, f_N)$. For the kinetic BGK case, $\partial_x f_{N+1}^{\text{Train}}(x_j,t_n) = \partial_x f_{N+1}^{\text{DVM}}$.

\begin{table}[ht]
    \centering
        \begin{tabular}{ |c | c| c| }
        \hline
        \textbf{Parameters} & \textbf{HME} & \textbf{Kinetic}\\
        \hline
        Training Data: & 100 ICs & 800 ICs\\
        Split Train/Test \%:&  80/20  & 80/20 \\
        Activation Function:& ReLU  & ReLU\\
        \hline
        Learning Rate: & Schedule: \texttt{StepLR} & Schedule: \texttt{OneCycleLR} \\
        & Rate: $10^{-7}$   &\quad Varies with count of  \\
        & Decay Factor: 0.5            &\quad moment input over epochs \\
        & Step Schedule: 100 epoch & \\
        \hline
        Optimizer:& \texttt{Adam} & \texttt{AdamW}\\
        \hline
        Loss Term: &  \multirow{2}{*}{$-f_N \partial_x u - \frac{f_{N-1}}{2} \partial_x \theta $}& \multirow{2}{*}{$ \partial_x m_{N+1}^{\text{DVM}}(x_j,t_n) $}  \\
           $\partial_x m_{N+1}^{\text{Train}}(x_j,t_n)$    &            &      \\    
        \hline
        \end{tabular}
        \caption{Summary of neural networks training parameters for both HME and kinetic closures.}
        \label{tab:NN Parameters Combined}
    \end{table}

\begin{table}[h!]
\centering
\begin{tabular}{ |c|c |c |c|  }
\hline
Moments  & Loss & Relative $L^2$ Error & Relative $L^2$ Error  \\
in Input &     & (Train)               &(Test) \\
\hline
5 & 1.49E-06 & 1.72E-01 & 1.76E-01   \\
\hline
6 & 4.38E-08 & 1.02E-01  & 1.05E-01  \\
\hline
7 & 1.22E-09 & 6.42E-02 & 6.55E-02 \\
\hline
8 & 1.56E-10 & 9.46E-02  & 9.52E-02  \\
\hline
\end{tabular}
\caption{Loss and errors, both in training and testing, for neural network of 9 layers, 128 neurons wide, and ReLU activation, at the end of training with 500 epoch on kinetic smooth data.}
\label{tab:train test error for prediction}
\end{table}

\FloatBarrier

{  Achieving good training errors for HME data was straight forward. However, several modifications were needed to achieve results with the kinetic data. These modifications are discussed in the following subsections. We outline the general training process in \Cref{algo:train_closure} for both HME and kinetic training data.}

\begin{algorithm}
\LinesNumbered
\KwData{The $M+1$ moments in primitive variables $\left(\rho, u, \theta, f_3, f_4, \dots, f_M, f_{M+1} \right)$ over the training time interval for every initial condition.}
\KwIn{A set of $M$ moments in primitive variables $\left(\rho, u, \theta, f_3, f_4, \dots, f_M \right)$.}
\KwOut{Coefficients on the gradient of the $M$ moments, $\mathcal{N}_0, \mathcal{N}_1, \dots, \mathcal{N}_M$}
\Begin(Prepare Training Set){
    \For{all time steps within training time range, in all training data}{
        {Load: $M+1$ moments, $f_{M+1}$ from training data} \\
        \nl Compute: $M+1$ moment gradients, $\partial_x f_{M+1}$ from training data
        }
    \Return{Prepared training set, $W^{Training}$}
}
\Begin(Training){
    \While{under number of training epoch}{
        { Sample $M$ moments, $M$ gradients and $\partial_x f_{M+1}$ from $ W^{Training}$}
        \BlankLine
        {\emph{ // Neural network defined in \Cref{sec:Neural Network Architecture}}} \\
        \nl{Evaluate Neural Network: $\mathcal{NN}(\rho, u, \theta, f_3, f_4, \dots, f_M) = \left(\mathcal{N}_0, \mathcal{N}_1, \dots, \mathcal{N}_M\right)$} \\
        \nl{Compute: $\partial_x f_{M+1}^{\text{NN}} = \sum^M_{i = 0} \mathcal{N}_i(\rho, u, \theta, f_{3}, \cdots, f_{M})\partial_x f_{i}$} \\
        \nl{Compute Loss: $\mathcal{L}(\partial_x f_{M+1}^{\text{NN}},\partial_x f_{M+1})$ }\\
        \nl{Back-propagation to optimize} 
        }
    }
\caption{{\sc {  Train Closure Model}}}
\label{algo:train_closure}
\end{algorithm}

\FloatBarrier
\subsubsection{\texttt{Adam} vs. \texttt{AdamW}}
The RTE work \cite{Huang2022-RTE1,Huang-RTE2,Huang-RTE3} used the \texttt{Adam} \cite{Adam} optimizer in PyTorch. This optimizer uses adaptive estimates for lower-order moments (mean, variance) of the gradient via exponential averaging and applying bias corrections.

The \texttt{AdamW}\cite{AdamW} optimizer, which has been also been used in recent moment closure work\cite{Li2023}, modifies the regularization that was implemented in \texttt{Adam}. In Loshchilov and Hutter\cite{AdamW}, it was argued that weight decay and $L^2$ regularization are not equivalent for adaptive gradient descent methods but are equivalent for stochastic gradient descent. \texttt{AdamW} separates moment estimates from the decay of the parameters, decoupling the optimization of the loss function from the weight decay. Empirically, \texttt{AdamW} has shown improved performance over wider array of hyper-parameter selection. 

By using \texttt{AdamW} over \texttt{Adam} for training our neural network closure, we were able to improve overall training error significantly as shown in \cref{tab:Adam vs AdamW}. {  This was needed for the kinetic training data.}

\begin{table}[h!]
\centering
\begin{tabular}{ |c|c |c | }
\hline
Moments  & \texttt{Adam} Relative         & \texttt{AdamW} Relative     \\
in Input & $L^2$ Error (Train)   & $L^2$ Error (Train) \\
\hline
6 & 1.23E-01 & 1.02E-01   \\
\hline
7 & 1.17E-01 & 6.42E-02 \\
\hline
8 & 1.06E-01 & 9.46E-02  \\
\hline
\end{tabular}
\caption{Comparison of relative training error using the \texttt{Adam} and \texttt{AdamW} optimizer in PyTorch for a 9 layer, 128 neuron wide, ReLU activated neural network with batch normalization. {  This is for training a NN closure on kinetic data and shows the improvement of switching optimizers.}}
\label{tab:Adam vs AdamW}
\end{table}

\FloatBarrier

\subsubsection{Adding Batch Normalization}
\FloatBarrier
To further improve training errors in the neural network closure, a batch normalization layer was added in the hidden layers between the linear and activation layers. Batch normalization \cite{BatchNormalization} is a commonly used method in deep learning. This process normalizes the batch input to the layer to have mean-zero, variance-one by tracking running mean and variance over the course of training. In \cref{fig:NN Architecture} the normalization layers can be thought of being contained in the layers with $\sigma$. Adding the batch normalization layer made optimal learning rate ranges more uniform across the number of moments input, as seen in \cref{fig:LRRTsub} as well as created a large region of decreasing loss, that stays relatively consistent over changing the number of moments that are input to the network. Also, this batch normalization layer also improves training errors during training.

\subsubsection{Learning Rate Cycling}
During training, there was concern that the neural network closure was settling in a local minimum rather than finding a global minimum. The initial implementation followed the RTE work \cite{Huang-RTE3} with a halving schedule for the learning rate hyper-parameter over a fixed number of epochs.

For the BGK neural network closure, we implemented a one-cycle learning rate schedule where the learning rate is raised and then lowered in a cycle following a cosine function \cite{OneCycle}. The learning rate changes occurs over each iteration of the optimizer, that is after each batch of training data rather than over epoch, which would be once through the entire training data set. A sample one-cycle learning rate schedule is show in \cref{fig:OCLRsub}.

This method requires finding optimal maximum learning rates that do not cause the training to diverge in an unrecoverable way from the minimum. This is done using the so called learning rate range test (LRRT) \cite{Smith2015-cyclic,Goodfellow-et-al-2016}. This process is useful for finding reasonable learning rate values for training in general. The LRRT varies the learning rate every training iteration over a range, orders of magnitude, of learning rates and records the loss. Usually revealed is a plateau where the learning rate is too small to approach the minimum and loss is nearly constant, followed by a region of decreasing loss that ends with a steep jump in loss as the learning rate becomes too large to resolve the minimum. By choosing a maximum learning rate that occurs at the end of the decreasing loss region and a minimum learning rate that occurs before the loss region, using a one cycle learning rate schedule.

There are draw backs to using this method. The LRRT must be done prior to training a given neural network with a given hyper parameter configuration. Also, if the maximum learning rate is chosen too far away from the end of the decrease, we have observed non-recoverable divergence away from a global minimum. However, this method does give a relatively good indication of where optimal learning rates are located.

\begin{figure}
    \subfloat[\label{fig:LRRTsub}]{
    \includegraphics[width=0.565\textwidth,valign=c]{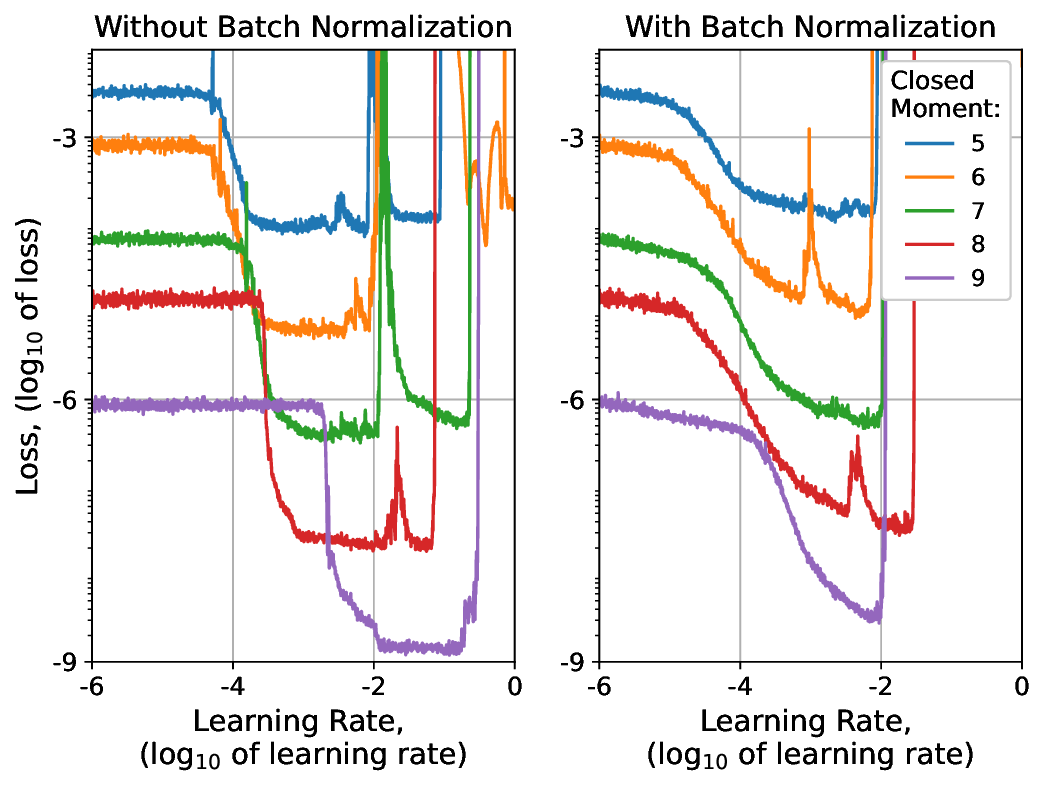}
    }
    \subfloat[\label{fig:OCLRsub}]{
    \includegraphics[width=0.4\textwidth,valign=c]{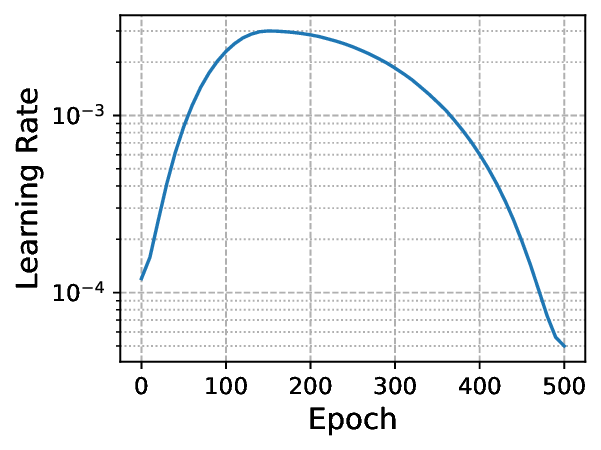}
    }
    \caption{Examples of learning rate range tests, with the effects of batch normalization, and a sample one-cycle learning rate schedule. \ref{fig:LRRTsub} Representative learning rate range for a 9 layer, 256 neuron wide network that is trained on kinetic smooth/wave initial conditions. The horizontal axes are the learning rate and the vertical axes are the loss. Each color represents the moment that is being closed by the network. On the left, the networks have no batch normalization layer, and on the right the networks do. We can see that batch normalization (right) gave overlapping regions of loss decrease that were spread over a wider range of learning rates compared to networks without. \ref{fig:OCLRsub} An example of the one cycle learning rate variation. The horizontal axis is the epoch of training and the vertical axis is the learning rate.}
\end{figure}

\FloatBarrier

\section{Numerical Method} 
\label{sec:Numerical Method}

In this section, we briefly review the first-order and high-order path-conservative numerical methods \cite{castro2017well} for solving the non-conservative hyperbolic moment system. We will show the formulation of the path conservative method. It is important to note we make use of a partially conservative transformation. This change of variables places the first $M$-th moments in conservative form, and the last moment in a non-conservative form as demonstrated in \cite{koellermeier2017numerical,koellermeier2021high}. The transformation to partially conserved variables greatly improves the performance of the path-conservative method.

For simplicity, we adopt a uniform spatial discretization of domain $[x_a, x_b]$,
$$
x_a = x_{\frac{1}{2}} < x_{\frac{3}{2}} \cdots < x_{N_x-\frac{1}{2}} < x_{N_x+\frac{1}{2}} = x_b, 
$$
where $x_{j-\frac{1}{2}} = x_a + j\Delta x, j = 1, 2, \cdots, N_x+1$ with mesh size $\Delta x = \frac{x_b-x_a}{N_x}$. The cells and the mid-points are denoted by
$$
I_j = [x_{j-\frac{1}{2}}, x_{j+\frac{1}{2}}], \quad x_j = \frac{1}{2}(x_{j-\frac{1}{2}} + x_{j+\frac{1}{2}}), \quad j = 1, 2, \cdots N_x.
$$
We denote $\overline{\mathbf{w}}^n_j = \frac{1}{\Delta x}\int^{x_{j+\frac{1}{2}}}_{x_{j-\frac{1}{2}}} \mathbf{w}(x, t^n)\ dx$ as the cell average value of $\mathbf{w}(x, t)$ at time $t^n = n\Delta t$ with $\Delta t$ being the time stepping size. 

\subsection{First order path-conservative scheme}

\label{sec:first-order-path-conser}

In the case of systems of hyperbolic conservation laws, when $A = \frac{\partial F}{\partial \mathbf{w}}$ is Jacobian matrix of a flux function $F(\mathbf{w})$, \cref{eq:primitive moment system matrix vector} can be written as
\begin{equation}
\label{eq:con-hcl}
\frac{\partial \mathbf{w}}{\partial t} + \frac{\partial F(\mathbf{w})}{\partial x} = 0 .   
\end{equation}
Then a first order conservative wave propagation method \cite{leveque1997wave} can be applied to \cref{eq:con-hcl},
\begin{equation} 
\label{eq:con-hcl-wp}
\overline{\mathbf{w}}^{n+1}_j = \overline{\mathbf{w}}^n_j - \frac{\Delta t}{\Delta x}\left(\mathcal{A}^+_{j-\frac{1}{2}}\Delta \mathbf{w}^n_{j-\frac{1}{2}} + \mathcal{A}^-_{j+\frac{1}{2}}\Delta \mathbf{w}^n_{j+\frac{1}{2}}\right) 
\end{equation}
with jump $\Delta \mathbf{w}^n_{j\pm\frac{1}{2}} = \mathbf{w}^{n, +}_{j\pm\frac{1}{2}} - \mathbf{w}^{n, -}_{j\pm\frac{1}{2}}$. Here $\mathbf{w}^{n, \pm}_{j\pm\frac{1}{2}}$ are the left and right limits of solution $\mathbf{w}^n$ at the cell boundaries $x_{j\pm\frac{1}{2}}$. In particular, $\mathbf{w}^{n, +}_{j-\frac{1}{2}}$ can be approximated using the cell average sequence from the neighboring cells $\{\overline{\mathbf{w}}^n_{j-q}, \cdots \overline{\mathbf{w}}^n_{j-p}\}$ for $q, p \geq 0$. In \cref{eq:con-hcl-wp}, it is sufficient for us to achieve first order accuracy taking $\mathbf{w}^{n, +}_{j-\frac{1}{2}} = \overline{\mathbf{w}}_j$ and $\mathbf{w}^{n, -}_{j-\frac{1}{2}} = \overline{\mathbf{w}}_{j-1}$, thus
$\Delta \mathbf{w}^n_{j-\frac{1}{2}} = \overline{\mathbf{w}}^n_{j} - \overline{\mathbf{w}}^n_{j-1}$ and $\Delta \mathbf{w}^n_{j+\frac{1}{2}} = \overline{\mathbf{w}}^n_{j+1} - \overline{\mathbf{w}}^n_{j}$ in \cref{eq:con-hcl-wp}. 

When $F(\mathbf{w}) = A\mathbf{w}$ is a linear function, $\mathcal{A} = A$, we have $A^{\pm}_{j\pm\frac{1}{2}}$ in \cref{eq:con-hcl-wp} defined by
\begin{equation}
A^+_{j\pm\frac{1}{2}} = R_{j\pm\frac{1}{2}}\Lambda^+_{j\pm\frac{1}{2}}R^{-1}_{j\pm\frac{1}{2}}, \quad	A^-_{j\pm\frac{1}{2}} = R_{j\pm\frac{1}{2}}\Lambda^-_{j\pm\frac{1}{2}}R^{-1}_{j\pm\frac{1}{2}}
\end{equation}
where $R_{j\pm\frac{1}{2}}$ is a $(M+1)\times (M+1)$ matrix storing the right eigenvectors of $A_{j\pm\frac{1}{2}}$, $\Lambda^{\pm}_{j\pm\frac{1}{2}} = 
\text{diag}\left((\lambda^1_{j+\frac{1}{2}})^{\pm}, (\lambda^2_{j+\frac{1}{2}})^{\pm}, \dots (\lambda^{M+1}_{j+\frac{1}{2}})^{\pm}\right)
$ with positive or negative eigenvalues of $A_{j\pm\frac{1}{2}}$ respectively. 
When $F(\mathbf{w})$ is nonlinear, the system shock speed at a discontinuity is given by Rankine-Hugoniot conditions.

For $A$ in \cref{eq:Grad's moment matrix,eq:moment system with ML closure} can not be expressed as a Jacobian matrix, we follow the definition of the non-conservative product $A(\mathbf{w})\mathbf{w}_x$ in \cite{dal1995definition}. For the left and right limits $\mathbf{w}^-, \mathbf{w}^+$ of $\mathbf{w}\in \Omega$ at the discontinuity, consider a family of Lipschitz continuous path $\phi(s; \mathbf{w}^-, \mathbf{w}^+): [0, 1]\times \Omega \times \Omega \rightarrow \Omega$ such that
$$
\phi(0; \mathbf{w}^-, \mathbf{w}^+) = \mathbf{w}^-, \quad \phi(1; \mathbf{w}^-, \mathbf{w}^+) = \mathbf{w}^+, \quad \phi(s; \mathbf{w}, \mathbf{w}) = \mathbf{w}
$$
where $\phi(s; \mathbf{w}^-, \mathbf{w}^+)$ can be seen as a parameterization of the integral curve connecting $\mathbf{w}^-, \mathbf{w}^+$. In \cite{koellermeier2017numerical}, there exist many different choices of paths $\phi$ connecting $W_L$ and $W_R$, such as a simple linear path
\begin{equation} \label{eq:path1}
\phi(s; \mathbf{w}^-, \mathbf{w}^+) = \mathbf{w}^- + s\cdot (\mathbf{w}^+ - \mathbf{w}^-), \quad s\in [0, 1];
\end{equation}
and polynomial paths $\phi^N_-$ and $\phi^N_+$:
\begin{equation} \label{eq:path2}
\phi^N_-(s; \mathbf{w}^-, \mathbf{w}^+) = \mathbf{w}^- + s^N\cdot (\mathbf{w}^+ - \mathbf{w}^-), \quad \phi^N_+(s; \mathbf{w}^-, \mathbf{w}^+) = \mathbf{w}^+ + (s-1)^N\cdot (\mathbf{w}^- - \mathbf{w}^+).
\end{equation}
In \cite{castro2017well}, given a family of paths $\phi$, a numerical scheme is said to be $\phi$-conservative if it can be written under the form
\begin{equation} \label{eq:path-conser}
\overline{\mathbf{w}}^{n+1}_j
= 
\overline{\mathbf{w}}^n_j 
- \frac{\Delta t}{\Delta x}\left(D^+_{j-\frac{1}{2}} + D^-_{j+\frac{1}{2}}\right), 
\end{equation}
where
$D^{\pm}_{j\pm\frac{1}{2}} = D^{\pm}(\mathbf{w}^{n, -}_{j\pm\frac{1}{2}}, \mathbf{w}^{n, +}_{j\pm\frac{1}{2}})$
with $D^-$ and $D^+$ being two continuous functions satisfying
\begin{equation} \label{eq:path-conser-consistency}
D^{\pm}(\mathbf{w}, \mathbf{w}) = 0, \quad \forall \mathbf{w}\in \Omega,
\end{equation}
and
\begin{multline} \label{eq:path-conser-jump}
D^-(\mathbf{w}^{n, -}_{j\pm\frac{1}{2}}, \mathbf{w}^{n, +}_{j\pm\frac{1}{2}}) 
+ 
D^+(\mathbf{w}^{n, -}_{j\pm\frac{1}{2}}, \mathbf{w}^{n, +}_{j\pm\frac{1}{2}}) 
= \\
\int^1_0 A(\phi(s; \mathbf{w}^{n, -}_{j\pm\frac{1}{2}}, \mathbf{w}^{n, +}_{j\pm\frac{1}{2}}))\frac{\partial \phi}{\partial s}(s; \mathbf{w}^{n, -}_{j\pm\frac{1}{2}}, \mathbf{w}^{n, +}_{j\pm\frac{1}{2}})\ ds.
\end{multline}
Consider a Roe linearization $\mathcal{A}_{\phi}(\mathbf{w}^-_{j\pm\frac{1}{2}}, \mathbf{w}^+_{j\pm\frac{1}{2}})$ satisfying
\begin{equation} \label{eq:groe}
\mathcal{A}_{\phi}\left(\mathbf{w}^-_{j\pm\frac{1}{2}}, \mathbf{w}^+_{j\pm\frac{1}{2}}\right)
\Delta \mathbf{w}_{j\pm\frac{1}{2}}
= \int^1_0 A(\phi(s; \mathbf{w}^-_{j\pm\frac{1}{2}}, \mathbf{w}^+_{j\pm\frac{1}{2}}))
\frac{\partial \phi}{\partial s}(s; \mathbf{w}^-_{j\pm\frac{1}{2}}, \mathbf{w}^+_{j\pm\frac{1}{2}}))\ ds.
\end{equation}
and let 
$
D^{\pm}(\mathbf{w}^{n, -}_{j\pm\frac{1}{2}}, \mathbf{w}^{n, +}_{j\pm\frac{1}{2}}) =  \mathcal{A}^{\pm}_{\phi}
\Delta \mathbf{w}_{j\pm\frac{1}{2}},
$
the first order path-conservative scheme for \cref{eq:primitive moment system matrix vector} is given as
\begin{equation} \label{eq:path-conser-roe}
\overline{\mathbf{w}}^{n+1}_j = \overline{\mathbf{w}}^n_j - 
\frac{\Delta t}{\Delta x}\left(\mathcal{A}^+_{\phi}\Delta \mathbf{w}_{j-\frac{1}{2}} + \mathcal{A}^-_{\phi}\Delta \mathbf{w}_{j+\frac{1}{2}}\right).
\end{equation}
When using \cref{eq:path1}, \cref{eq:groe} reduces
\begin{equation} \label{eq:groe2}
\mathcal{A}_{\phi}\left(\mathbf{w}^-_{j\pm\frac{1}{2}}, \mathbf{w}^+_{j\pm\frac{1}{2}}\right) 
= 
\int^1_0 A(\phi(s; \mathbf{w}^-_{j\pm\frac{1}{2}}, \mathbf{w}^+_{j\pm\frac{1}{2}}))\ ds,
\end{equation}
which could be computed via
\begin{equation} \label{eq:groe4}
\mathcal{A}_{\phi}\left(\mathbf{w}^-_{j\pm\frac{1}{2}}, \mathbf{w}^+_{j\pm\frac{1}{2}}\right) \approx \sum^p_{i = 1} \omega_iA(\phi(s_i, \mathbf{w}^-_{j\pm\frac{1}{2}}, \mathbf{w}^+_{j\pm\frac{1}{2}}))
\end{equation}
using high order quadrature rule with points $\{s_i\}$ and the corresponding weights $\{\omega_i\}$. 
When using \cref{eq:path2}, $\mathcal{A}_{\phi}$ can be computed as
\begin{align} \label{eq:groe3}
\mathcal{A}_{\phi}\left(\mathbf{w}^-_{j\pm\frac{1}{2}}, \mathbf{w}^+_{j\pm\frac{1}{2}}\right) 
        &= \int^1_0 A(\phi(s; \mathbf{w}^-_{j\pm\frac{1}{2}}, \mathbf{w}^+_{j\pm\frac{1}{2}}))\cdot n \cdot s^{n-1}\ ds \\
		&\approx \sum^p_{i = 1} \omega_iA(\phi(s_i; \mathbf{w}^-_{j\pm\frac{1}{2}}, \mathbf{w}^+_{j\pm\frac{1}{2}}))\cdot n \cdot s^{n-1}_i
\end{align}
via high order quadrature rules.

\subsection{High order path-conservative scheme}

\label{sec:high-order-path-conser}

We again start from \cref{eq:con-hcl}, a conservative semi-discrete method is given as
\begin{equation} \label{eq:high-order-con}
\mathbf{w}'_j(t) = -\frac{1}{\Delta x}\left(G_{j+\frac{1}{2}}(\mathbf{w}^-_{j+\frac{1}{2}}(t), \mathbf{w}^+_{j+\frac{1}{2}}(t)) - G_{j-\frac{1}{2}}(\mathbf{w}^-_{j-\frac{1}{2}}(t), \mathbf{w}^+_{j-\frac{1}{2}}(t))\right)
\end{equation}
with the numerical flux
\begin{subequations}
\label{eq:high-order-con-flux}
\begin{align} 
G_{j\pm\frac{1}{2}}(\mathbf{w}^-_{j\pm\frac{1}{2}}(t), \mathbf{w}^+_{j\pm\frac{1}{2}}(t)) 
    &= F(\mathbf{w}^+_{j\pm\frac{1}{2}}(t)) - \mathcal{A}^+_{j\pm\frac{1}{2}}
    \Delta \mathbf{w}_{j\pm\frac{1}{2}}(t)
    \label{eq:high-order-con-flux1}\\
	&= F(\mathbf{w}^-_{j\pm\frac{1}{2}}(t)) + \mathcal{A}^-_{j\pm\frac{1}{2}}
    \Delta \mathbf{w}_{j\pm\frac{1}{2}}(t)
    \label{eq:high-order-con-flux2}.
\end{align}
\end{subequations}
Note that \cref{eq:high-order-con-flux1,eq:high-order-con-flux2} are obtained from
$$
F(\mathbf{w}^+_{j\pm\frac{1}{2}}(t)) - F(\mathbf{w}^-_{j\pm\frac{1}{2}}(t)) 
= \mathcal{A}_{j\pm\frac{1}{2}}
\Delta \mathbf{w}_{j\pm\frac{1}{2}}(t)
.
$$
Plug \cref{eq:high-order-con-flux} in 
\cref{eq:high-order-con}, we have
\begin{subequations}
\label{eq:high-order-path-conser}
\begin{align} 
\mathbf{w}'_j(t) 
    &= -\frac{1}{\Delta x}\left[\mathcal{A}^+_{j-\frac{1}{2}}
    \Delta \mathbf{w}_{j-\frac{1}{2}}(t)
		+ \mathcal{A}^-_{j+\frac{1}{2}}
        \Delta \mathbf{w}_{j+\frac{1}{2}}(t)
		+ F(\mathbf{w}^-_{j-\frac{1}{2}}(t)) - F(\mathbf{w}^+_{j-\frac{1}{2}}(t)) \right] \label{eq:high-order-path-conser1}\\
	&= -\frac{1}{\Delta x}\left[\mathcal{A}^+_{j-\frac{1}{2}}
    \Delta \mathbf{w}_{j-\frac{1}{2}}(t)
		+ \mathcal{A}^-_{j+\frac{1}{2}}
        \Delta \mathbf{w}_{j+\frac{1}{2}}(t)
		+ \int^{x_{j+\frac{1}{2}}}_{x_{j-\frac{1}{2}}} A(P_j(x))\frac{d}{dx} P_j(x)\ dx\right]\label{eq:high-order-path-conser2}
\end{align}
\end{subequations}
with smooth function $P_j(x)$ on cell $I_j$ satisfying
\begin{equation}
\lim_{x\rightarrow x^+_{j-\frac{1}{2}}} P_j(x) = \mathbf{w}^+_{j-\frac{1}{2}}(t)	\quad	
\lim_{x\rightarrow x^-_{j+\frac{1}{2}}} P_j(x) = 
\mathbf{w}^-_{j+\frac{1}{2}}(t).
\end{equation}
\cref{eq:high-order-path-conser} leads to a $\phi$-conservative scheme for \cref{eq:primitive moment system matrix vector}
\begin{equation} \label{eq:high-order-path-conser-phi}
\mathbf{w}'_j(t) = -\frac{1}{\Delta x}\left[\mathcal{A}^+_{\phi}
    \Delta \mathbf{w}_{j-\frac{1}{2}}(t)
		+ \mathcal{A}^-_{\phi}
        \Delta \mathbf{w}_{j+\frac{1}{2}}(t)
		+ \int^{x_{j+\frac{1}{2}}}_{x_{j-\frac{1}{2}}} A(P_j(x))\frac{d}{dx} P_j(x)\ dx\right]
\end{equation}
which can be written as 
$$
\mathbf{w}'_j(t)
= 
- \frac{1}{\Delta x}\left(D^+_{j-\frac{1}{2}} + D^-_{j+\frac{1}{2}}\right) 
$$
with
\begin{align}
D^+_{j-\frac{1}{2}} 
    &= \mathcal{A}^+_{\phi}
    \Delta \mathbf{w}_{j-\frac{1}{2}}(t)
    + \int^{x_j}_{x_{j-\frac{1}{2}}} A(P_j(x))\frac{d}{dx} P_j(x)\ dx \\
D^-_{j+\frac{1}{2}} 
    &= \mathcal{A}^-_{\phi}
    \Delta \mathbf{w}_{j+\frac{1}{2}}(t)
    + \int^{x_{j+\frac{1}{2}}}_{x_j} A(P_j(x))\frac{d}{dx} P_j(x)\ dx 
\end{align}
and
\begin{align*}
D^-_{j+\frac{1}{2}} + D^+_{j+\frac{1}{2}} 
    &= \int^{x_{j+\frac{1}{2}}}_{x_j} A(P_j(x))\frac{d}{dx} P_j(x)\ dx  \\
	&+ \mathcal{A}_{\phi}
    \Delta \mathbf{w}_{j+\frac{1}{2}}(t)    
    \\
	&+ \int^{x_{j+1}}_{x_{j+\frac{1}{2}}} A(P_{j+1}(x))\frac{d}{dx} P_{j+1}(x)\ dx \\
	&= \int^{x_{j+\frac{1}{2}}}_{x_j} A(P_j(x))\frac{d}{dx} P_j(x)\ dx  \\
    &+ \int^1_0 A(\phi(s; \mathbf{w}^-_{j+\frac{1}{2}}(t), \mathbf{w}^+_{j+\frac{1}{2}}(t)))\frac{\partial \phi}{\partial s}(s; \mathbf{w}^-_{j+\frac{1}{2}}(t), \mathbf{w}^+_{j+\frac{1}{2}}(t))\ ds  \\
	&+ \int^{x_{j+1}}_{x_{j+\frac{1}{2}}} A(P_{j+1}(x))\frac{d}{dx} P_{j+1}(x)\ dx.
\end{align*}
In our implementation, the integral term $\int^{x_{j+\frac{1}{2}}}_{x_{j-\frac{1}{2}}} A(P_j(x))\frac{d}{dx} P_j(x)\ dx$ in \cref{eq:high-order-path-conser-phi} is approximated via a sixth order Gauss Lobatto quadrature rule
\begin{equation}
\int^{x_{j+\frac{1}{2}}}_{x_{j-\frac{1}{2}}} A(P_j(x))\frac{d}{dx} P_j(x)\ dx 
\approx 
\sum^4_{i = 1} w_i A(P_j(s_i\Delta x + x_j))\frac{d}{dx} P_j(s_i\Delta x + x_j).
\end{equation}
with quadrature nodes $\{s_i\}^4_{i=1} \in [-\frac{1}{2}, \frac{1}{2}]$ and their associated weights $\{w_i\}^4_{i=1}$. More specifically, the approximated point values 
$\{\mathbf{w}^+_{j-\frac{1}{2}} = P_j(x_{j-\frac{1}{2}}) = P_j(s_1\Delta x + x_j), \\ P_j(s_2\Delta x + x_j), P_j(s_3\Delta x + x_j), \mathbf{w}^-_{j+\frac{1}{2}} = P_j(x^-_{j+\frac{1}{2}}) = P_j(s_4\Delta x + x_j)\}$ 
over $I_j$ can be reconstructed using a fifth order accurate WENO method~\cite{shu1988efficient} from the cell averages $\{\overline{\mathbf{w}}_i, i = j-2, j,  j+2\}$. $\{\frac{d}{dx} P_j(x)| x = s_i\Delta x + x_j, i = 1, \cdots 4\}$ can then be approximated by the derivative of polynomial interpolating through $\{P_j(x)| x = s_i\Delta x + x_j, i = 1, \cdots 4\}$ over $I_j$.

Note that \cref{eq:high-order-path-conser-phi} becomes a system of ODEs after the spatial discretization to the RHS. For the time integration of \cref{eq:high-order-path-conser-phi} on the LHS, we use a standard third-order SSP Runge-Kutta method \cite{gottlieb2001strong}.

\section{Numerical Results}
\label{sec:Numerical Results}
In this section we present numerical calculations using the neural network closures, for both HME trained and kinetic trained, for the moment system and associated errors.

{  For the kinetic trained closure, we are making an assumption on the form of the closure which is an inherent model error. For HME trained closure, the closure model is exact so the only error present is related to the machine learning errors of training and stochastic randomness associated with it. The purpose of the presenting the HME results is to demonstrate that the machine learning approach to the closure problem is viable.}

{  An outline of process to compute solutions using the trained neural network closure is given in \Cref{algo:numerical_experiment}.}

\begin{algorithm}
\LinesNumbered
\DontPrintSemicolon
\KwIn{\\
\quad Moments in primitive variables: $\mathbf{w}^0 =\left(\rho^0, u^0, \theta^0, f_3^0, f_4^0, \dots, f_M^0 \right)$ at $t=0$.\\ 
\quad Trained Neural Network Closure: $\mathcal{NN}$}
\KwOut{Moments in primitive variables: $ \mathbf{w}^{T_{final}} = \left(\rho, u, \theta, f_3, f_4, \dots, f_M \right)$ at $t=T_{final}$.}
\setcounter{AlgoLine}{0}
{Form $A^{\text{ML}}$ with last row $ \gets \mathcal{NN}(\mathbf{w}^0) = \begin{pmatrix}
    a_0 & a_1 & a_2 & a_3 & \cdots & a_{M-1} & a_{M}
\end{pmatrix}$ \Cref{eq:moment system with ML closure} }\\
\nl {$\mathbf{w}^* \gets \mathbf{w}^0$} \\
\nl \While{$t < T_{final}$}{
\emph{ // First Stage}\\
\nl Compute: $A^{\text{ML}}$ using $\mathcal{NN}(\mathbf{w}^*)$ for last row.\\
\nl Compute: maximum wave speeds and CFL from $A^{\text{ML}}$\\
\nl Compute: $\texttt{rhs}^1$, the right-hand side of $\frac{\partial{\mathbf{w}}}{\partial t} = - A^{\text{ML}}\frac{\partial{\mathbf{w}}}{\partial x} + \textbf{Q} $ for $\mathbf{w}^*$ using non-conservative methods in \Cref{sec:Numerical Method}.\\
\nl$\mathbf{w}^1 \gets \mathbf{w^*} - \Delta t \cdot \mathfrak{a}_{1,0}\cdot \texttt{rhs}^1$\\
\nl$t \gets t + \mathfrak{c}_1 \cdot \Delta t$\BlankLine
\emph{// Second stage} \\
\nl Compute: $A^{\text{ML}}$ using $\mathcal{NN}(\mathbf{w}^1)$ for last row.\\
\nl Compute: maximum wave speeds and CFL from $A^{\text{ML}}$ \\
\nl Compute: $\texttt{rhs}^2$, the right-hand side of $\frac{\partial{\mathbf{w}}}{\partial t} = - A^{\text{ML}}\frac{\partial{\mathbf{w}}}{\partial x} + \textbf{Q} $ for $\mathbf{w}^1$ using non-conservative methods in \Cref{sec:Numerical Method}.\\
\nl$\mathbf{w}^2 \gets \mathbf{w^*} - \Delta t \cdot \left( \mathfrak{a}_{2,0}\cdot \texttt{rhs}^1 +  \mathfrak{a}_{2,1}\cdot \texttt{rhs}^2 \right)$\\
\nl$t \gets t + \mathfrak{c}_2 \cdot \Delta t$\BlankLine
\emph{// Third stage} \\
\nl Compute: $A^{\text{ML}}$ using $\mathcal{NN}(\mathbf{w}^2)$ for last row.\\
\nl Compute: maximum wave speeds and CFL from $A^{\text{ML}}$ \\
\nl Compute: $\texttt{rhs}^3$, the right-hand side of $\frac{\partial{\mathbf{w}}}{\partial t} = - A^{\text{ML}}\frac{\partial{\mathbf{w}}}{\partial x} + \textbf{Q} $ for $\mathbf{w}^2$ using non-conservative methods in \Cref{sec:Numerical Method}.\\
\nl$\mathbf{w}^3 \gets \mathbf{w^*} - \Delta t \cdot \left( \mathfrak{b}_{0}\cdot \texttt{rhs}^1 +  \mathfrak{b}_{1}\cdot \texttt{rhs}^2 +  \mathfrak{b}_{2}\cdot \texttt{rhs}^3\right)$\\
\nl$t\gets t+\Delta t$ \\
\nl$\mathbf{w}^* \gets \mathbf{w}^3$
}
\caption{{\sc {  Computing with Trained Closure Model}} \\ 
{  Solve the non-conservative moment system $$ \frac{\partial{\mathbf{w}}}{\partial t} + A^{\text{ML}}\frac{\partial{\mathbf{w}}}{\partial x} = \textbf{Q} $$ using trained neural network closure to complete last row of $ A^{\text{ML}}$, using RK3 and the non-conservative schemes described in \Cref{sec:Numerical Method}. The variables $\mathfrak{a},\mathfrak{b},\mathfrak{c}$ are the coefficients for the Runge-Kutta scheme.}}
\label{algo:numerical_experiment}
\end{algorithm}

\FloatBarrier

\subsection{Realization of HME Model using Neural Network}

\FloatBarrier

The neural network recovered the HME model as shown in \cref{fig:HME samples moments t=10}, for wave training data, and in \cref{fig:HME Mix samples moments t=10}, for mix training data, {  with the data generated as described in \cref{subsec:Training Considerations and Techniques}} . For wave trained HME closure, there was very good agreement during numerical prediction with wave initial conditions as seen in \cref{fig:HME samples moments t=10}. When mix initial conditions were used for prediction with a closure trained on wave data, the closure had difficulty recovering moments as seen in \cref{fig:HME smooth train/mix ic samples moments t=10}. However, when using mix initial conditions with a closure trained on 50\% wave, 50\% mix data, the HME model was recovered well, as seen in \cref{fig:HME Mix samples moments t=10}. {  The robust performance of the NN closure trained on HME data, as demonstrated in \cref{fig:HME samples moments t=10} and \cref{fig:HME Mix samples moments t=10}, is universal and representative of its effectiveness across various Knudsen number regimes.}

Recovery of the HME model by the trained network occurred over the range of Knudsen numbers tested for both smooth/wave and mix data as seen in Figures \ref{fig:HME l2 error vs kn t=10} and \ref{fig:HME mix abs. l2 error vs kn t=10}.
 
\subsection{Kinetic Trained Neural Network Closure Results}

\FloatBarrier

After training neural networks on moment data generated as described in 
\Cref{sec:Training Setup and Neural Network Architecture}, we then use this closure in the solver, using the FORCE method \cite{koellermeier2021high,castro2017well}, to predict moments using initial conditions that were not part of the training data set but generated in the same manner. The results are at $t=0.3$, which is beyond the training set time interval of $t=0.05$ to $t=0.15$. \Cref{fig:samples moments t=0.3} show individual samples of the first three moments corresponding to the primitive variables of density, velocity, and temperature across the range of Knudsen numbers at the various times. These compare the computed moments using the neural network closure and the moment data at matching times using the discrete velocity method in \Cref{sec:Training Setup and Neural Network Architecture}. 
{  In addition, to ensure the reproducibility of numerical experiments, in \Cref{tab:prediction sample parameters}, we present the parameters for smooth, wave kinetic initial conditions that are used in the numerical predictions shown in \Cref{fig:samples moments t=0.3}.}
\Cref{fig:l2 error vs kn t=0.3} show the relative $L^2$ error in the primitive variables at the computed times for all initial conditions in this test set.

Generally, we see that within the training set time interval we have good agreement between between neural network closure moments and this test set for the highly collisional range of Knudsen numbers, $[10^{-3},10^{-1})$. When computing with trained closures within the training time window, we have very good agreement across the range of Knudsen numbers, with relative $L^2$ error not exceeding 10\%. As we move beyond the training window, results in the transition and free streaming range of Knudsen numbers, $[10^{-1},10^1]$, vary greatly as seen in \cref{fig:l2 error vs kn t=0.3}. {  There also appears to be larger error in the macroscopic velocity, $u$, as well as clustering of moments of different order. This might be due to the parity of the moments, where $\rho$ and $\theta$ are even powered and $u$ is odd powered. This could be an avenue of future investigation.}

{  To demonstrate the advantage of our hyperbolic NN closure, we train a non-hyperbolic, kinetic NN closure that learns the gradient coefficients as in our previous work \cite{Huang2022-RTE1}. The prediction results with the Knudsen number $10^{-3}$ and $t=0.3$ are shown in \Cref{fig:non-hyperbolic vs hyperbolic}. We observe that the profiles of the lower-order moments $\rho$, $u$, and $\theta$ for the non-hyperbolic closure agree well with the kinetic data. At the same time, the oscillations appear in the higher-order moments $f_3$ and $f_4$. This is due to the loss of hyperbolicity in the ML model. As a comparison, the predictions of our hyperbolic NN closure align well with the kinetic data for all the moments. Additionally, it is worth noting that the oscillations increase over time and worsen with increasing Knudsen numbers, though these results are omitted here due to the page limit.
}

{  We also numerically test the Galilean invariance of our ML model. We train the neural network using the dataset with the macroscopic velocity of $u=0$. Then we use the ML model to predict the solution profiles with an initial condition with $u=0.1$. The profiles of density, velocity, and temperature at $t=0.3$ are presented in \Cref{fig:Galilean invariance example}. The results of the ML model agree well with those from the kinetic equation. This shows that our neural network generates closures which preserves Galilean invariance.}

\begin{figure}[ht]
    \centering
    \includegraphics[scale=0.40]{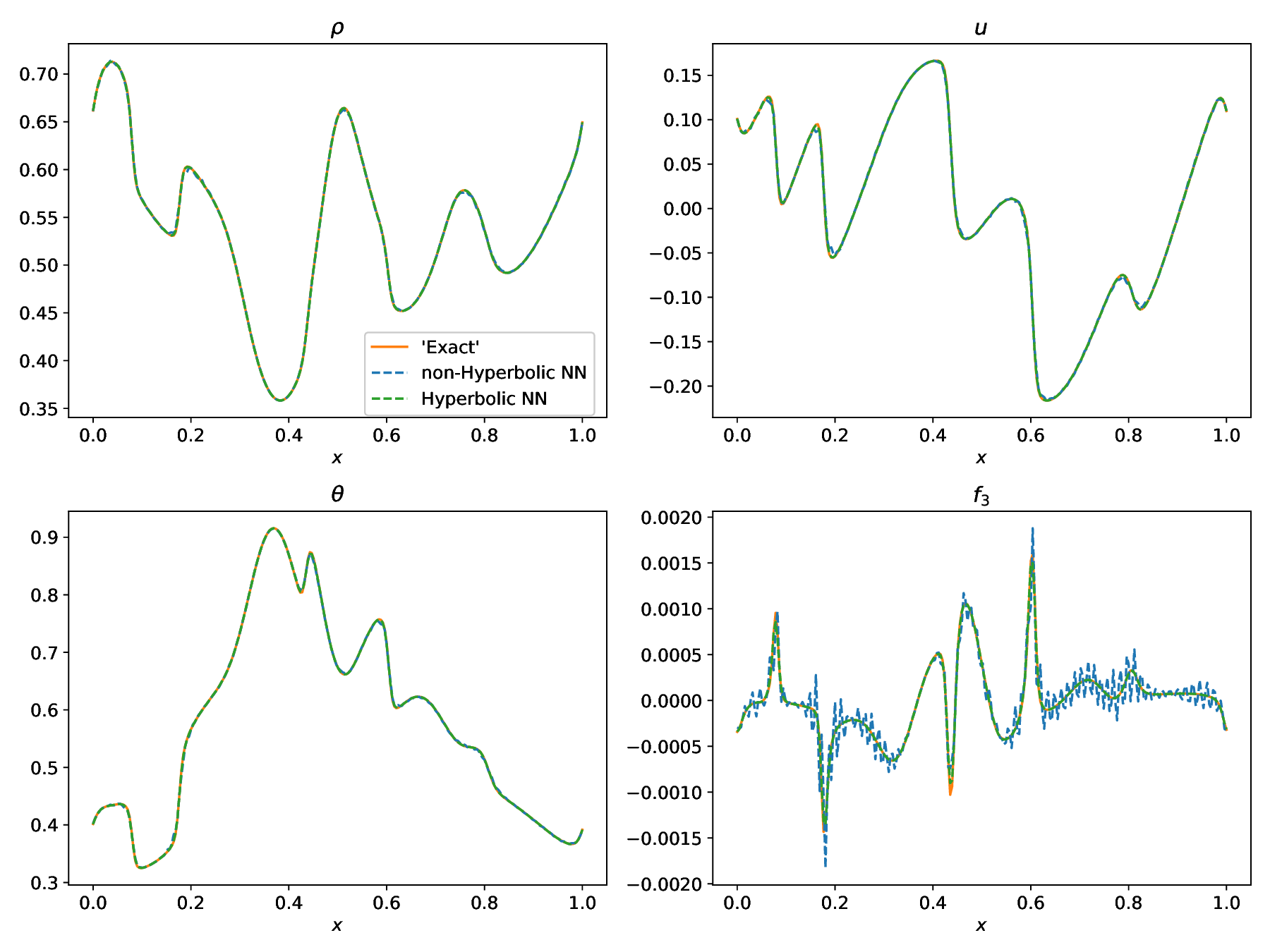}  
    \caption{{  Moments calculated using a kinetic trained NN without hyperbolic method and with the hyperbolic method. This example has a Knudsen number of 0.001. This prediction is at $t=0.3$. 
    }}
    \label{fig:non-hyperbolic vs hyperbolic}
\end{figure}

\begin{figure}[ht]
    \centering
    \includegraphics[scale=0.5]{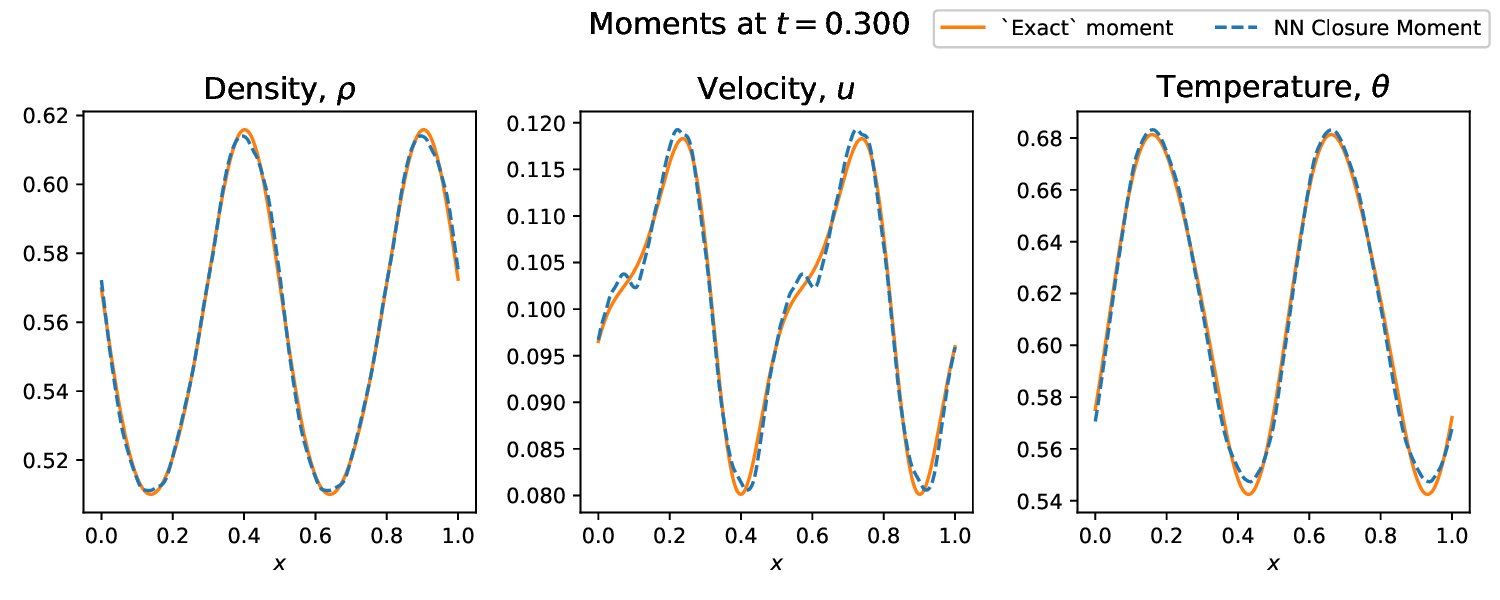}   
    \caption{{  Galilean invariance using a kinetic trained closure for the first three moments, $\rho,u,\theta$, for Knudsen number of 0.063. The NN closure was trained with a macroscopic velocity of $u=0$ and was used to predict for an initial condition with $u=0.1$. This prediction is at $t=0.3$. 
    }}
    \label{fig:Galilean invariance example}
\end{figure}

Predicting beyond the training set time interval with the kinetic trained closure, we see the highly collisional range of Knudsen numbers continue to have good agreement with the test set moments for smooth initial conditons. However, in the transition and free streaming Knudsen regimes we see wide variation in the quality of the solutions, with a max relative $L^2$ error of approximately $5.0$. But there are also solutions that are very well behaved and with relative $L^2$ error of approximately $0.02$. This can be seen in \cref{fig:l2 error vs kn t=0.3}. {  Preliminary investigation of the non-uniform performance at intermediate Knudsen number, i.e.\ $[0.1, 1.0]$, indicates evidence of over-fitting within the NN.} Adding dropout and increasing $L^2$-regularization with builtin PyTorch modules and optimizer parameters show some improvement in error and eliminated solver failures on some initial conditions. This will be explored in future work.

Investigating the computed kinetic trained moment closures in the Knudsen numbers above $10^{-1}$, we generally observe high frequency oscillations develop beyond the training set time interval. This was not observed in the HME NN closures nor in the earlier RTE work which both had good agreement over all collision regimes. These oscillations did not occur for all initial conditions. During our investigation we studied the maximum distance between eigenvalues computed from the NN closure, which hints at the potential source of the oscillations. Comparing cases with similar Knudsen numbers, we observe that well behaved predictions had well separated eigenvalues through time, while predictions that saw oscillations develop saw minimum separation between eigenvalues shrink significantly before causing instability.

{  We compare both the time required for calculation and the relative errors in the primitive variables for a kinetic trained NN closure with 8 total moments and the rational HME models with varying number of moments for a Knudsen number of 0.0466 in \Cref{tab:Model Compare}. The initial condition is sampled from smooth generated data. As a base for comparison, we use a DVM calculation as the reference for time and as the comparison for errors. The kinetic trained NN closure is competitive with HME in computational time while providing improved accuracy in the primitive variables. The accuracy of the solution in the primitive variables fell between 10 and 12 moment HME model and time to solution comparable to 8 and 10 moment HME model. This shows modest improvements for Knudsen value approaching the transition regime. Our future work will focus on improving accuracy in the transition and free streaming limit Knudsen regimes. This only accounts for the computation time provided a trained NN closure; the time to create training sets and train the associated NN are much longer.}

{  As a point of interest, we formulated and tested classic Riemann problems and compared with kinetic NN closures trained on both smooth/wave and 50-50 smooth and mixed data. For small Knudsen numbers, below 0.01, the ML model accurately captures the appropriate dynamics, including how the initial discontinuity spreads. Interestingly, the added mixed training data provided superior performance in resolving the higher moments. For the transition regime where the Knudsen number approaches 0.1, the model captures bulk dynamics but exhibits spurious oscillations. Generalizations that accurately capture the dynamics in the transition regime and beyond are subjects of our ongoing and future work.}

\begin{figure}[ht]
    \centering
    \includegraphics[trim={0 55 0 0},clip,scale=0.4]{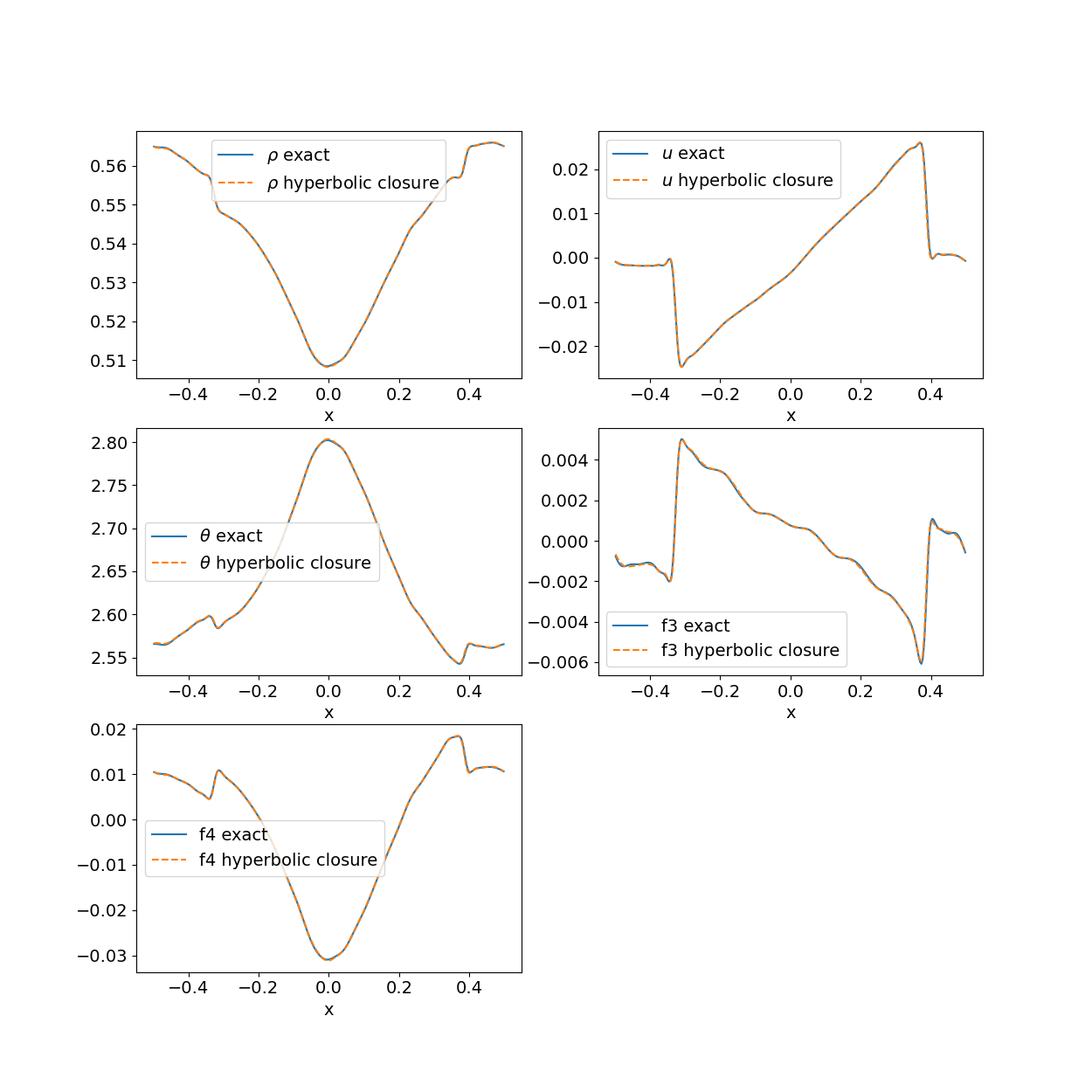}
    \caption{HME trained closure: Sample of five moments ($\rho, u, \theta, f_3, f_4$) for Knudsen number $\tau=1$ at $t=10$, beyond the training data time range. trained on smooth, wave data then used to predict moments on smooth initial conditions.}
    \label{fig:HME samples moments t=10}
\end{figure}

\begin{figure}
    \subfloat[\label{fig:HME smooth train/mix ic samples moments t=10}]{
    \includegraphics[trim={0 0 465 0},clip,width=0.45\textwidth,valign=c]{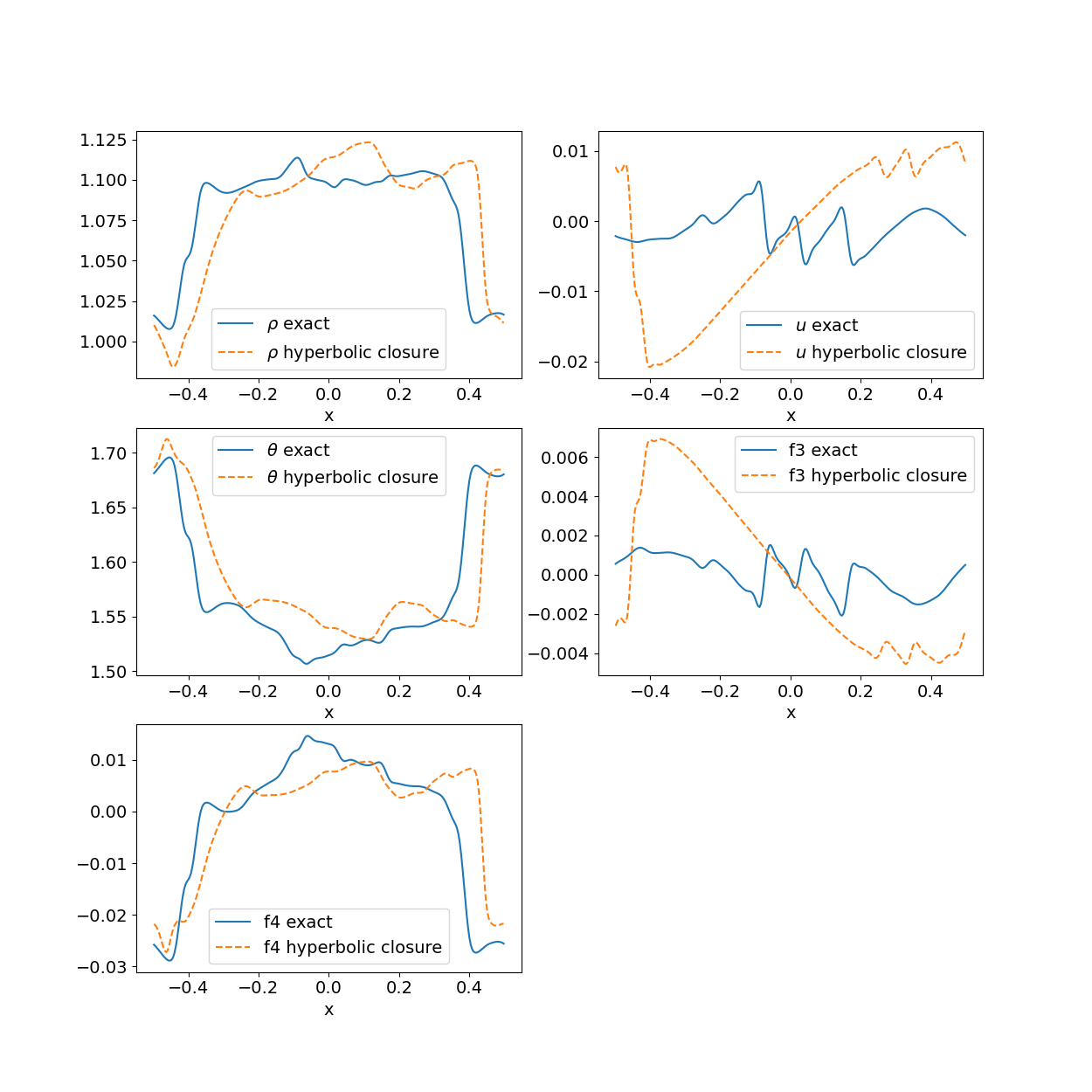}
    }
    \subfloat[\label{fig:HME Mix samples moments t=10}]{
    \includegraphics[trim={0 0 465 0},clip,width=0.45\textwidth,valign=c]{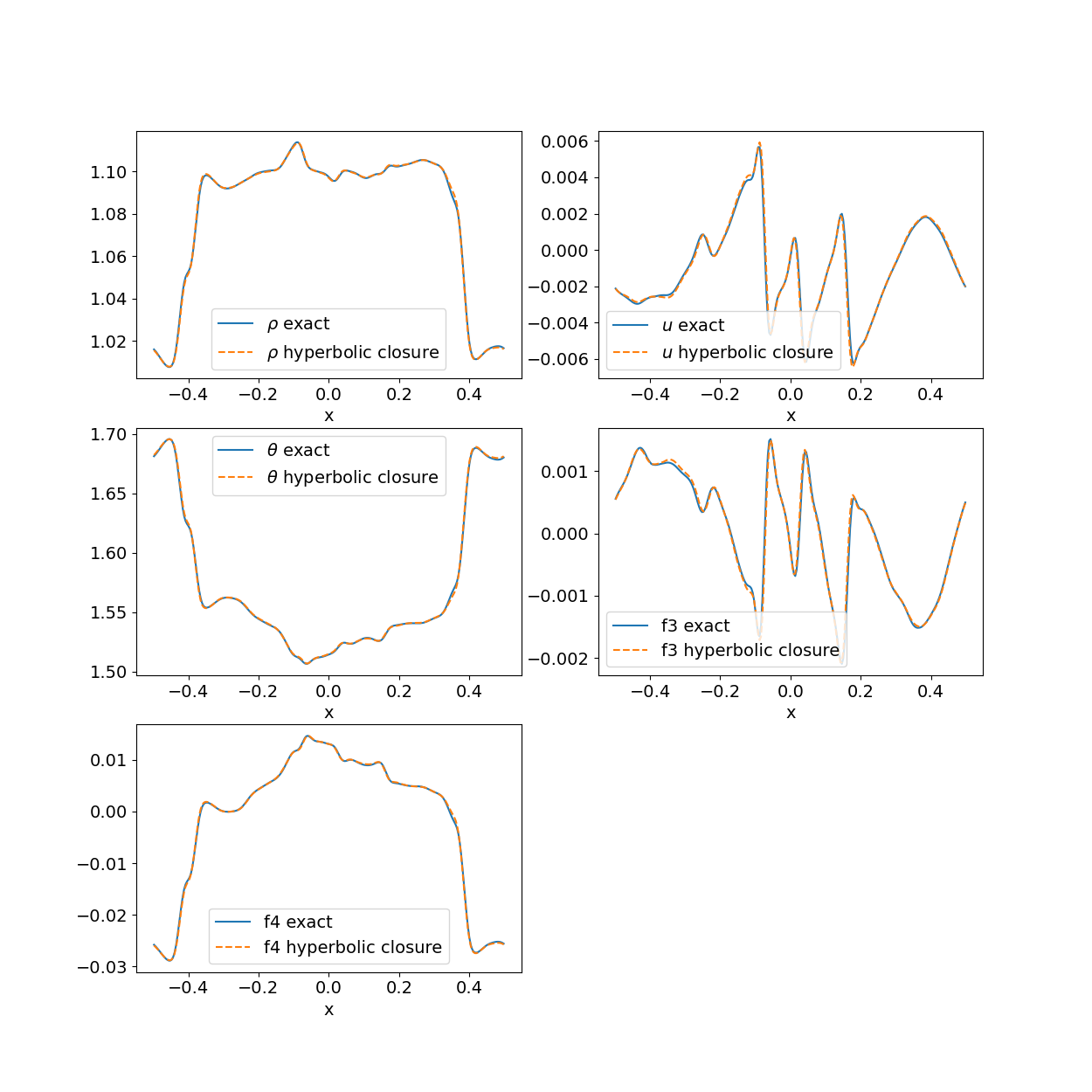}
    }
    \caption{HME trained closure: Sample from five moment model ($\rho,\theta, f_4$) for Knudsen number $\tau=1$ at time $t=10$, beyond the training data time range. \ref{fig:HME smooth train/mix ic samples moments t=10} HME closure trained on smooth, wave data then used to predict moments on mixed initial conditions. \ref{fig:HME Mix samples moments t=10} HME closure trained on a training set consisting of 50\% smooth data and 50\% mix data.}
\end{figure}

\begin{figure}
    \subfloat[\label{fig:HME l2 error vs kn t=10}]{
    \includegraphics[trim={0 0 45 40},clip,width=0.48\textwidth,valign=c]{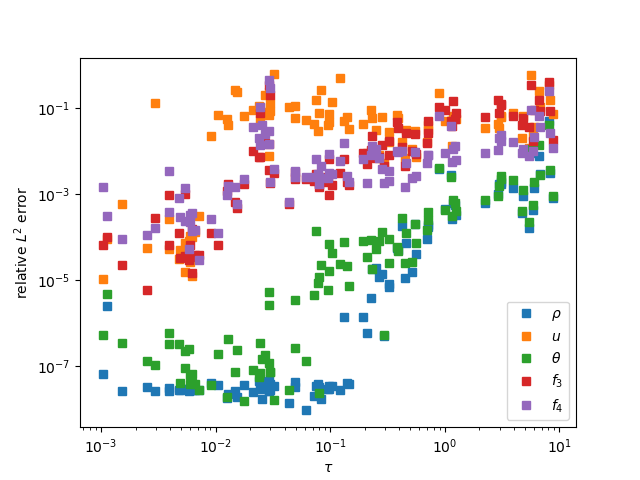}
    }
    \subfloat[\label{fig:HME mix abs. l2 error vs kn t=10}]{
    \includegraphics[trim={0 0 45 40},clip,width=0.48\textwidth,valign=c]{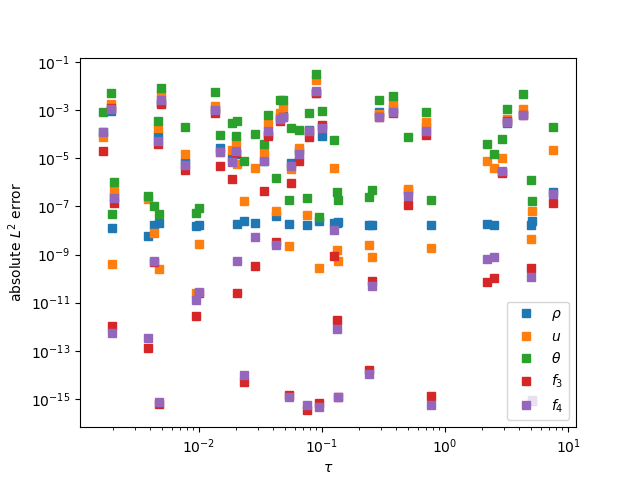}
    }
    \caption{HME trained closure: Errors for Knudsen numbers at $t=10$, beyond the training data time range. Left: relative $L^2$ error in $\rho, u, \theta, f_3, f_4$ at $t=10$, beyond the training data time range for HME trained neural network closure on smooth/wave data. Right: absolute $L^2$ error in $\rho, u, \theta, f_3, f_4$ at $t=10$, beyond the training data time range for HME trained neural network closure on mix data.}
\end{figure}

\FloatBarrier

\begin{figure}[ht]
    \includegraphics[trim={0 0 0 5},clip,scale=0.38]{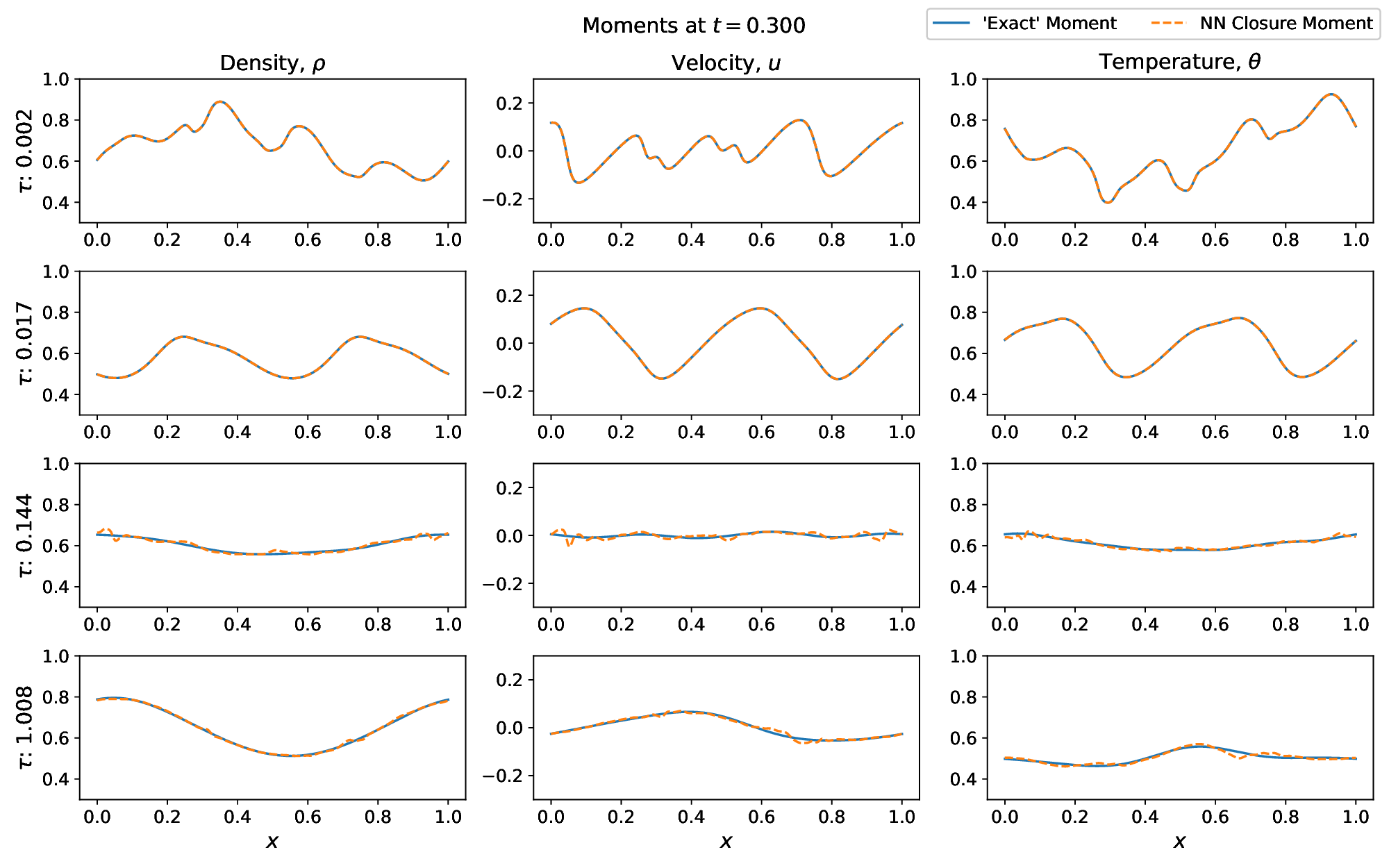}
    \caption{Kinetic trained closure: First three moments, $\rho, u, \theta$, for various Knudsen numbers at $t=0.3$, beyond the training data time range from neural network of 9 layers, 128 neuron wide, using ReLU with batch normalization for smooth initial conditions. Each column corresponds to a moment going left to right: $\rho, u, \theta$. Each row is a different Knudsen number, going top to bottom: 0.002, 0.017, 0.144, 1.008.}
    \label{fig:samples moments t=0.3}
\end{figure}

\begin{table}[ht]
\centering
\resizebox{\columnwidth}{!}{%
\begin{tabular}{|c|llll|llll|l|}
\hline
\multicolumn{1}{|l|}{\multirow{2}{*}{\begin{tabular}[c]{@{}l@{}}Knudsen \\ Number\end{tabular}}} &
  \multicolumn{4}{c|}{Density} &
  \multicolumn{4}{c|}{Temperature} &
  \multicolumn{1}{c|}{\multirow{2}{*}{\begin{tabular}[c]{@{}c@{}}Mixing\\ $\alpha$\end{tabular}}} \\ \cline{2-9}
\multicolumn{1}{|l|}{} &
  \multicolumn{1}{c|}{$a_\rho$} &
  \multicolumn{1}{c|}{$b_\rho$} &
  \multicolumn{1}{c|}{$k_\rho$} &
  \multicolumn{1}{c|}{$\phi_\rho$} &
  \multicolumn{1}{c|}{$a_\theta$} &
  \multicolumn{1}{c|}{$b_\theta$} &
  \multicolumn{1}{c|}{$k_\theta$} &
  \multicolumn{1}{c|}{$\phi_\theta$} &
  \multicolumn{1}{c|}{} \\ \hline
\multirow{2}{*}{0.001554} &
  \multicolumn{1}{l|}{0.261888} &
  \multicolumn{1}{l|}{0.270198} &
  \multicolumn{1}{l|}{4} &
  0.347195 &
  \multicolumn{1}{l|}{0.228545} &
  \multicolumn{1}{l|}{0.665582} &
  \multicolumn{1}{l|}{1} &
  2.708624 &
  0.438211 \\ \cline{2-10} 
 &
  \multicolumn{1}{l|}{0.270198} &
  \multicolumn{1}{l|}{0.627907} &
  \multicolumn{1}{l|}{1} &
  5.5624 &
  \multicolumn{1}{l|}{0.223145} &
  \multicolumn{1}{l|}{0.622083} &
  \multicolumn{1}{l|}{4} &
  1.573187 &
  0.596517 \\ \hline
\multirow{2}{*}{0.016515} &
  \multicolumn{1}{l|}{0.222878} &
  \multicolumn{1}{l|}{0.655877} &
  \multicolumn{1}{l|}{1} &
  0.484743 &
  \multicolumn{1}{l|}{0.213417} &
  \multicolumn{1}{l|}{0.574442} &
  \multicolumn{1}{l|}{3} &
  3.280153 &
  0.066066 \\ \cline{2-10} 
 &
  \multicolumn{1}{l|}{0.229181} &
  \multicolumn{1}{l|}{0.585005} &
  \multicolumn{1}{l|}{2} &
  4.482850 &
  \multicolumn{1}{l|}{0.253143} &
  \multicolumn{1}{l|}{0.639633} &
  \multicolumn{1}{l|}{2} &
  5.905965 &
  0.981557 \\ \hline
\multirow{2}{*}{0.143742} &
  \multicolumn{1}{l|}{0.246574} &
  \multicolumn{1}{l|}{0.574064} &
  \multicolumn{1}{l|}{3} &
  3.953562 &
  \multicolumn{1}{l|}{0.254433} &
  \multicolumn{1}{l|}{0.253392} &
  \multicolumn{1}{l|}{1} &
  4.545444 &
  0.589359 \\ \cline{2-10} 
 &
  \multicolumn{1}{l|}{0.250002} &
  \multicolumn{1}{l|}{0.686808} &
  \multicolumn{1}{l|}{4} &
  4.205230 &
  \multicolumn{1}{l|}{0.253392} &
  \multicolumn{1}{l|}{0.525255} &
  \multicolumn{1}{l|}{3} &
  4.694920 &
  0.384366 \\ \hline
\multirow{2}{*}{1.008160} &
  \multicolumn{1}{l|}{0.224907} &
  \multicolumn{1}{l|}{0.614791} &
  \multicolumn{1}{l|}{2} &
  5.343272 &
  \multicolumn{1}{l|}{0.251982} &
  \multicolumn{1}{l|}{0.569712} &
  \multicolumn{1}{l|}{2} &
  0.829738 &
  0.34663 \\ \cline{2-10} 
 &
  \multicolumn{1}{l|}{0.219768} &
  \multicolumn{1}{l|}{0.68036} &
  \multicolumn{1}{l|}{1} &
  1.102718 &
  \multicolumn{1}{l|}{0.250576} &
  \multicolumn{1}{l|}{0.514169} &
  \multicolumn{1}{l|}{1} &
  4.491626 &
  0.639318 \\ \hline
\end{tabular}%
}
\caption{{  Parameters for smooth, wave kinetic initial conditions, described in \Cref{subsec:training data preparation}, that are used in the numerical predictions shown in \Cref{fig:samples moments t=0.3}.}}
\label{tab:prediction sample parameters}
\end{table}

\begin{figure}[ht]
    \centering
    \includegraphics[trim={0 0 0 0},clip,scale=0.55]{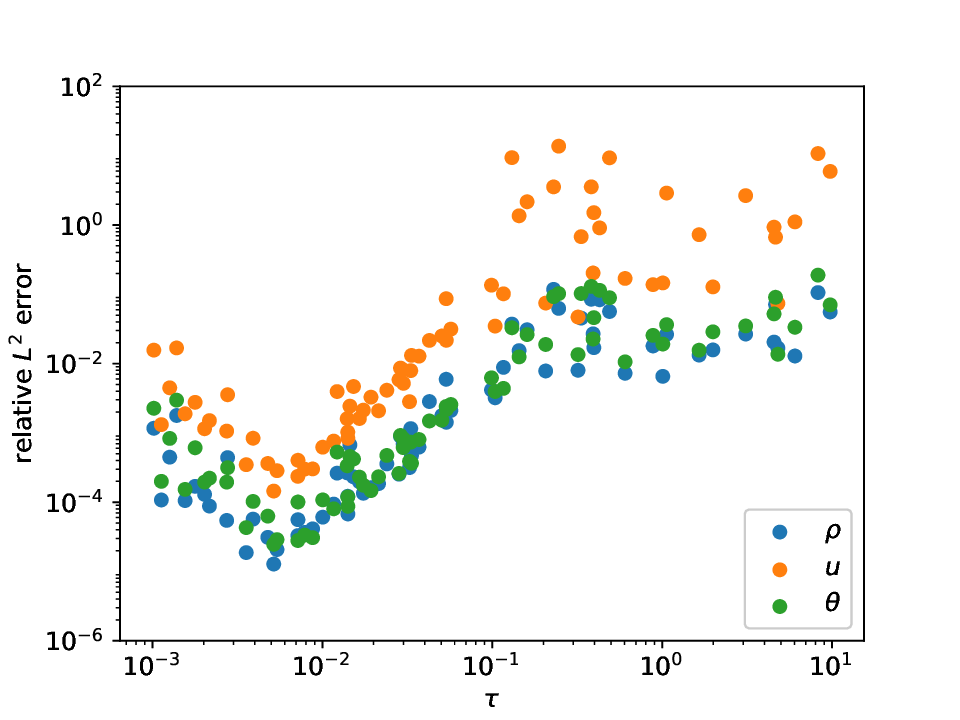}
    \caption{Kinetic trained closure: Relative $L^2$ error in $\rho, u, \theta$ at $t=0.3$, beyond the training data time range from neural network of 9 layers, 128 neuron wide, using ReLU with batch normalization for smooth initial conditions. Of the 100 initial conditions, 26 cases failed in the solver.}
    \label{fig:l2 error vs kn t=0.3}
\end{figure}

\begin{table}[ht]
\centering
\begin{tabular}{|c|c|c|c|c|c|c|}
\hline
Model                          & DVM    & HME-6  & HME-8 & HME-10 & HME-12 & NN-8 \\
\hline
Time (s)                       &  198   &  86    & 107   & 174   & 237    & 149  \\
\hline
Relative Time                  &  1.00  &  0.43  & 0.54  & 0.88  & 1.20   & 0.75  \\
\hline
Rel. $L^2$ Error: $\rho$  &  -     &  1.13\%  & 0.51\%  & 0.28\% & 0.15\% &  0.19\%  \\
\hline
Rel. $L^2$ Error: $u$     &  -     &  7.46\%  & 3.18\%  & 2.49\% & 1.94\% &  1.51\%  \\
\hline
Rel. $L^2$ Error: $\theta$ &  -    &  1.43\%  & 0.60\%  & 0.42\% & 0.26\% &  0.35\%  \\
\hline
\end{tabular}
\caption{{  Comparison of runtime and relative errors for HME models (no NN closure) with 6, 8, 10, and 12 moments and the kinetic trained NN closure with 8 total moments. Run-times and errors are computed against the DVM model. The HME model and the kinetic NN closure models were run in \texttt{python}, while the DVM model was run in \texttt{MATLAB}. Comparison made at $t=0.3$ for a case with Knudsen number of $0.0466$ for smooth data.}}
\label{tab:Model Compare}
\end{table}

\FloatBarrier

\section{Conclusion}
\label{sec:Conclusion}
We have developed and applied hyperbolicity preserving neural networks to the moment closure of the Grad expansion of the BGK model. We have demonstrated that when a local neural network moment closure is trained on data from the HME closure of the Grad expansion, the model robustly recovers HME both inside and outside of the training time window. Further, when the model is trained on moments of the full kinetic model, using DVM, the local neural network moment closure captures kinetic effects over a wider range of Knudsen numbers within the training window and beyond.  We also demonstrate the utility of introducing non-smooth data into the training processes for prediction beyond the training window.  

For kinetic BGK training data we have demonstrated neural network closures for smooth initial conditions over a range of Knudsen numbers. In the fluid regime, i.e.\ Knudsen numbers between $10^{-3}$ and $10^{-1}$, we have good agreement within and beyond the training time window. For transitional and free streaming regimes, i.e.\ Knudsen numbers between $10^{-1}$ and $10$, we demonstrate good agreement within the training time interval. Outside of the training time window, while relative error is small we see non-uniform performance in predictions. This may be due to our model not preserving the asymptotic limit of long time dynamics, which was preserved in the symmeterizer approach in earlier RTE work \cite{Huang-RTE2}.

Future work would be to improve and expand the kinetic BGK closures ability to generalize beyond training data set, especially in the in and above the transition regime. Possible approaches could include better NN regularization, greatly increasing the training set size with respect to Knudsen numbers, using transfer learning from HME NN closures to the kinetic closure, and implementing NN architecture that have regularization through low rank tensor decomposition. Another possibility is the development of a symmeterizer for the system as in the earlier RTE work \cite{Huang-RTE2}, but this would be non-trivial for the BGK model compared to RTE.

\section*{Acknowledgements}

{  The authors would like to thank the anonymous reviewers for their invaluable comments and questions which improved the quality of this paper.}
This work was supported in part by Michigan State University through computational resources provided by the Institute for Cyber-Enabled Research (ICER).
The authors would like to thank the various funding agencies for their support of this work.

\bibliographystyle{siamplain}
\bibliography{bibliography/reference}{}

\begin{thebibliography}{10}

\bibitem{Ascher1997}
{\sc U.~M. Ascher, S.~J. Ruuth, and R.~J. Spiteri}, {\em Implicit-explicit runge-kutta methods for time-dependent partial differential equations}, Applied Numerical Mathematics, 25 (1997), pp.~151--167, \url{https://doi.org/10.1016/s0168-9274(97)00056-1}, \url{https://doi.org/10.1016/s0168-9274(97)00056-1}.

\bibitem{Bhatnagar1954}
{\sc P.~L. Bhatnagar, E.~P. Gross, and M.~Krook}, {\em A model for collision processes in gases. i. small amplitude processes in charged and neutral one-component systems}, Physical Review, 94 (1954), pp.~511--525, \url{https://doi.org/10.1103/physrev.94.511}, \url{https://doi.org/10.1103/physrev.94.511}.

\bibitem{broadwell1964study}
{\sc J.~E. Broadwell}, {\em Study of rarefied shear flow by the discrete velocity method}, Journal of Fluid Mechanics, 19 (1964), pp.~401--414.

\bibitem{Cai2013}
{\sc Z.~Cai, Y.~Fan, and R.~Li}, {\em Globally hyperbolic regularization of grad's moment system in one-dimensional space}, Communications in Mathematical Sciences, 11 (2013), pp.~547--571, \url{https://doi.org/10.4310/cms.2013.v11.n2.a12}, \url{https://doi.org/10.4310/cms.2013.v11.n2.a12}.

\bibitem{Cai2015}
{\sc Z.~Cai, Y.~Fan, and R.~Li}, {\em A framework on moment model reduction for kinetic equation}, {SIAM} Journal on Applied Mathematics, 75 (2015), pp.~2001--2023, \url{https://doi.org/10.1137/14100110x}, \url{https://doi.org/10.1137/14100110x}.

\bibitem{castro2017well}
{\sc M.~J. Castro, T.~M. de~Luna, and C.~Par{\'e}s}, {\em {Well-balanced schemes and path-conservative numerical methods}}, in Handbook of Numerical Analysis, vol.~18, Elsevier, 2017, pp.~131--175.

\bibitem{Charalampopoulos2022}
{\sc A.~Charalampopoulos, S.~H. Bryngelson, T.~Colonius, and T.~P. Sapsis}, {\em Hybrid quadrature moment method for accurate and stable representation of non-gaussian processes applied to bubble dynamics}, Philosophical Transactions of the Royal Society A: Mathematical, Physical and Engineering Sciences, 380 (2022), \url{https://doi.org/10.1098/rsta.2021.0209}, \url{http://dx.doi.org/10.1098/rsta.2021.0209}.

\bibitem{dal1995definition}
{\sc G.~Dal~Maso, P.~G. Lefloch, and F.~Murat}, {\em {Definition and weak stability of nonconservative products}}, Journal de math{\'e}matiques pures et appliqu{\'e}es, 74 (1995), pp.~483--548.

\bibitem{NIST:DLMF}
{\em {\it NIST Digital Library of Mathematical Functions}}.
\newblock \url{https://dlmf.nist.gov/}, Release 1.1.12 of 2023-12-15, \url{https://dlmf.nist.gov/}.
\newblock F.~W.~J. Olver, A.~B. {Olde Daalhuis}, D.~W. Lozier, B.~I. Schneider, R.~F. Boisvert, C.~W. Clark, B.~R. Miller, B.~V. Saunders, H.~S. Cohl, and M.~A. McClain, eds.

\bibitem{Elouafi2009}
{\sc M.~Elouafi and A.~D.~A. Hadj}, {\em A recursion formula for the characteristic polynomial of hessenberg matrices}, Applied Mathematics and Computation, 208 (2009), pp.~177--179, \url{https://doi.org/10.1016/j.amc.2008.11.027}, \url{https://doi.org/10.1016/j.amc.2008.11.027}.

\bibitem{Goodfellow-et-al-2016}
{\sc I.~Goodfellow, Y.~Bengio, and A.~Courville}, {\em Deep Learning}, MIT Press, 2016.
\newblock \url{http://www.deeplearningbook.org}.

\bibitem{gottlieb2001strong}
{\sc S.~Gottlieb, C.-W. Shu, and E.~Tadmor}, {\em Strong stability-preserving high-order time discretization methods}, {SIAM} Review, 43 (2001), pp.~89--112, \url{https://doi.org/10.1137/s003614450036757x}, \url{https://doi.org/10.1137/s003614450036757x}.

\bibitem{Grad1949}
{\sc H.~Grad}, {\em On the kinetic theory of rarefied gases}, Communications on Pure and Applied Mathematics, 2 (1949), pp.~331--407, \url{https://doi.org/10.1002/cpa.3160020403}, \url{https://doi.org/10.1002/cpa.3160020403}.

\bibitem{Han-WeinanE-2019}
{\sc J.~Han, C.~Ma, Z.~Ma, and W.~E}, {\em Uniformly accurate machine learning-based hydrodynamic models for kinetic equations}, Proceedings of the National Academy of Sciences, 116 (2019), pp.~21983--21991, \url{https://doi.org/10.1073/pnas.1909854116}, \url{https://doi.org/10.1073/pnas.1909854116}.

\bibitem{Huang2022-RTE1}
{\sc J.~Huang, Y.~Cheng, A.~J. Christlieb, and L.~F. Roberts}, {\em Machine learning moment closure models for the radiative transfer equation {I}: Directly learning a gradient based closure}, Journal of Computational Physics, 453 (2022), p.~110941, \url{https://doi.org/10.1016/j.jcp.2022.110941}, \url{https://doi.org/10.1016/j.jcp.2022.110941}.

\bibitem{Huang-RTE3}
{\sc J.~Huang, Y.~Cheng, A.~J. Christlieb, and L.~F. Roberts}, {\em Machine learning moment closure models for the radiative transfer equation {III}: Enforcing hyperbolicity and physical characteristic speeds}, Journal of Scientific Computing, 94 (2022), \url{https://doi.org/10.1007/s10915-022-02056-7}, \url{http://dx.doi.org/10.1007/s10915-022-02056-7}.

\bibitem{Huang-RTE2}
{\sc J.~Huang, Y.~Cheng, A.~J. Christlieb, L.~F. Roberts, and W.-A. Yong}, {\em Machine learning moment closure models for the radiative transfer equation {II}: Enforcing global hyperbolicity in gradient-based closures}, Multiscale Modeling \& Simulation, 21 (2023), pp.~489--512, \url{https://doi.org/10.1137/21M1423956}, \url{https://doi.org/10.1137/21M1423956}, \url{https://arxiv.org/abs/https://doi.org/10.1137/21M1423956}.

\bibitem{BatchNormalization}
{\sc S.~Ioffe and C.~Szegedy}, {\em Batch normalization: Accelerating deep network training by reducing internal covariate shift}, in Proceedings of the 32nd International Conference on International Conference on Machine Learning - Volume 37, ICML'15, JMLR.org, 2015, p.~448–456.

\bibitem{jiang1996efficient}
{\sc G.-S. Jiang and C.-W. Shu}, {\em Efficient implementation of weighted {ENO} schemes}, Journal of computational physics, 126 (1996), pp.~202--228.

\bibitem{sesame-report}
{\sc J.~Johnson and S.~Lyon}, {\em {SESAME}: The {L}os {A}lamos {N}ational {L}aboratory equations of state database}, Tech. Report LA-UR-92-3407, Los Alamos National Laboratory, 1992.

\bibitem{Adam}
{\sc D.~P. Kingma and J.~Ba}, {\em Adam: A method for stochastic optimization}, 2014, \url{https://doi.org/10.48550/ARXIV.1412.6980}, \url{https://arxiv.org/abs/1412.6980}.

\bibitem{koellermeier2021high}
{\sc J.~Koellermeier and M.~Castro}, {\em {High-Order Non-Conservative Simulation of Hyperbolic Moment Models in Partially-Conservative Form}}, East Asian Journal on Applied Mathematics, 11 (2021), pp.~435--467.

\bibitem{Koellermeier2014}
{\sc J.~Koellermeier, R.~P. Schaerer, and M.~Torrilhon}, {\em A framework for hyperbolic approximation of kinetic equations using quadrature-based projection methods}, Kinetic and Related Models, 7 (2014), pp.~531--549, \url{https://doi.org/10.3934/krm.2014.7.531}, \url{https://doi.org/10.3934/krm.2014.7.531}.

\bibitem{Koellermeier2014-AIPConf}
{\sc J.~Koellermeier and M.~Torrilhon}, {\em Hyperbolic moment equations using quadrature-based projection methods}, in {AIP} Conference Proceedings, {AIP} Publishing {LLC}, 2014, \url{https://doi.org/10.1063/1.4902651}, \url{https://doi.org/10.1063/1.4902651}.

\bibitem{koellermeier2017numerical}
{\sc J.~Koellermeier and M.~Torrilhon}, {\em {Numerical study of partially conservative moment equations in kinetic theory}}, Communications in Computational Physics, 21 (2017), pp.~981--1011.

\bibitem{Koellermeier2018-QBME2d}
{\sc J.~Koellermeier and M.~Torrilhon}, {\em Two-dimensional simulation of rarefied gas flows using quadrature-based moment equations}, Multiscale Modeling {\&} Simulation, 16 (2018), pp.~1059--1084, \url{https://doi.org/10.1137/17m1147548}, \url{https://doi.org/10.1137/17m1147548}.

\bibitem{leveque1997wave}
{\sc R.~J. LeVeque}, {\em Wave propagation algorithms for multidimensional hyperbolic systems}, Journal of Computational Physics, 131 (1997), pp.~327--353, \url{https://doi.org/10.1006/jcph.1996.5603}, \url{https://doi.org/10.1006/jcph.1996.5603}.

\bibitem{Levermore1996}
{\sc C.~D. Levermore}, {\em Moment closure hierarchies for kinetic theories}, Journal of Statistical Physics, 83 (1996), p.~1021–1065, \url{https://doi.org/10.1007/bf02179552}, \url{http://dx.doi.org/10.1007/BF02179552}.

\bibitem{Li2023}
{\sc Z.~Li, B.~Dong, and Y.~Wang}, {\em Learning invariance preserving moment closure model for boltzmann{\textendash}{BGK} equation}, Communications in Mathematics and Statistics,  (2023), \url{https://doi.org/10.1007/s40304-022-00331-5}, \url{https://doi.org/10.1007/s40304-022-00331-5}.

\bibitem{li2023solving}
{\sc Z.~Li, Y.~Wang, H.~Liu, Z.~Wang, and B.~Dong}, {\em Solving boltzmann equation with neural sparse representation}, 2023, \url{https://arxiv.org/abs/2302.09233}.

\bibitem{AdamW}
{\sc I.~Loshchilov and F.~Hutter}, {\em Decoupled weight decay regularization}, 2017, \url{https://doi.org/10.48550/ARXIV.1711.05101}, \url{https://arxiv.org/abs/1711.05101}.

\bibitem{sesame-McHardy2018}
{\sc J.~D. McHardy}, {\em An introduction to the theory and use of {SESAME} equations of state}, Tech. Report LA-14503, Office of Scientific and Technical Information (OSTI), Dec. 2018, \url{https://doi.org/10.2172/1487368}, \url{http://dx.doi.org/10.2172/1487368}.

\bibitem{MILLER2022111541}
{\sc S.~T. Miller, N.~V. Roberts, S.~D. Bond, and E.~C. Cyr}, {\em Neural-network based collision operators for the boltzmann equation}, Journal of Computational Physics, 470 (2022), p.~111541, \url{https://doi.org/https://doi.org/10.1016/j.jcp.2022.111541}, \url{https://www.sciencedirect.com/science/article/pii/S0021999122006039}.

\bibitem{Naris2004}
{\sc S.~Naris, D.~Valougeorgis, F.~Sharipov, and D.~Kalempa}, {\em Discrete velocity modelling of gaseous mixture flows in {MEMS}}, Superlattices and Microstructures, 35 (2004), pp.~629--643, \url{https://doi.org/10.1016/j.spmi.2004.02.025}, \url{https://doi.org/10.1016/j.spmi.2004.02.025}.

\bibitem{NEURIPS2019_9015-Pytorch}
{\sc A.~Paszke, S.~Gross, F.~Massa, A.~Lerer, J.~Bradbury, G.~Chanan, T.~Killeen, Z.~Lin, N.~Gimelshein, L.~Antiga, A.~Desmaison, A.~Kopf, E.~Yang, Z.~DeVito, M.~Raison, A.~Tejani, S.~Chilamkurthy, B.~Steiner, L.~Fang, J.~Bai, and S.~Chintala}, {\em Pytorch: An imperative style, high-performance deep learning library}, in Advances in Neural Information Processing Systems 32, Curran Associates, Inc., 2019, pp.~8024--8035, \url{http://papers.neurips.cc/paper/9015-pytorch-an-imperative-style-high-performance-deep-learning-library.pdf}.

\bibitem{porteous2021data}
{\sc W.~A. Porteous, M.~P. Laiu, and C.~D. Hauck}, {\em Data-driven, structure-preserving approximations to entropy-based moment closures for kinetic equations}, arXiv preprint arXiv:2106.08973,  (2021).

\bibitem{pmlr-v162-schotthofer22a}
{\sc S.~Schotth{\"o}fer, T.~Xiao, M.~Frank, and C.~Hauck}, {\em Structure preserving neural networks: A case study in the entropy closure of the boltzmann equation}, in Proceedings of the 39th International Conference on Machine Learning, K.~Chaudhuri, S.~Jegelka, L.~Song, C.~Szepesvari, G.~Niu, and S.~Sabato, eds., vol.~162 of Proceedings of Machine Learning Research, PMLR, 17--23 Jul 2022, pp.~19406--19433, \url{https://proceedings.mlr.press/v162/schotthofer22a.html}.

\bibitem{schotthofer2022neural}
{\sc S.~Schotth{\"o}fer, T.~Xiao, M.~Frank, and C.~D. Hauck}, {\em Neural network-based, structure-preserving entropy closures for the boltzmann moment system}, arXiv preprint arXiv:2201.10364,  (2022).

\bibitem{shu1988efficient}
{\sc C.-W. Shu and S.~Osher}, {\em Efficient implementation of essentially non-oscillatory shock-capturing schemes}, Journal of Computational Physics, 77 (1988), pp.~439--471, \url{https://doi.org/10.1016/0021-9991(88)90177-5}, \url{https://doi.org/10.1016/0021-9991(88)90177-5}.

\bibitem{Smith2015-cyclic}
{\sc L.~N. Smith}, {\em Cyclical learning rates for training neural networks}, 2015, \url{https://doi.org/10.48550/ARXIV.1506.01186}, \url{https://arxiv.org/abs/1506.01186}.

\bibitem{OneCycle}
{\sc L.~N. Smith and N.~Topin}, {\em Super-convergence: Very fast training of neural networks using large learning rates}, 2017, \url{https://doi.org/10.48550/ARXIV.1708.07120}, \url{https://arxiv.org/abs/1708.07120}.

\bibitem{Stangeby1990}
{\sc P.~Stangeby and G.~McCracken}, {\em Plasma boundary phenomena in tokamaks}, Nuclear Fusion, 30 (1990), pp.~1225--1379, \url{https://doi.org/10.1088/0029-5515/30/7/005}, \url{https://doi.org/10.1088/0029-5515/30/7/005}.

\bibitem{Struchtrup2005}
{\sc H.~Struchtrup}, {\em Macroscopic Transport Equations for Rarefied Gas Flows}, Springer Berlin Heidelberg, 2005, \url{https://doi.org/10.1007/3-540-32386-4}, \url{https://doi.org/10.1007/3-540-32386-4}.

\bibitem{Taitano2021}
{\sc W.~Taitano, B.~Keenan, L.~Chacón, S.~Anderson, H.~Hammer, and A.~Simakov}, {\em An eulerian vlasov-fokker–planck algorithm for spherical implosion simulations of inertial confinement fusion capsules}, Computer Physics Communications, 263 (2021), p.~107861, \url{https://doi.org/10.1016/j.cpc.2021.107861}, \url{http://dx.doi.org/10.1016/j.cpc.2021.107861}.

\bibitem{Taitano2018}
{\sc W.~T. Taitano, A.~N. Simakov, L.~Chacón, and B.~Keenan}, {\em Yield degradation in inertial-confinement-fusion implosions due to shock-driven kinetic fuel-species stratification and viscous heating}, Physics of Plasmas, 25 (2018), \url{https://doi.org/10.1063/1.5024402}, \url{http://dx.doi.org/10.1063/1.5024402}.

\bibitem{Tseng2006}
{\sc K.-C. Tseng, J.-S. Wu, and I.~D. Boyd}, {\em Simulations of re-entry vehicles by using {DSMC} with chemical-reaction module}, in 14th {AIAA}/{AHI} Space Planes and Hypersonic Systems and Technologies Conference, American Institute of Aeronautics and Astronautics, June 2006, \url{https://doi.org/10.2514/6.2006-8084}, \url{https://doi.org/10.2514/6.2006-8084}.

\end{thebibliography}

\end{document}